\documentclass[10pt]{article} 
\usepackage[accepted]{tmlr}

\usepackage{amsmath,amsfonts,amsthm,bm}

\newcommand\norm[1]{\left\lVert#1\right\rVert}
\newcommand\bigoh[1]{\mathcal{O}\left(#1\right)}

\newcommand{\chevron}[1]{\left\langle #1 \right \rangle}
\newcommand{\expval}[0]{\mathbb{E} }

\newcommand{\pderiv}[2]{\frac{\partial #1}{\partial #2}}
\newcommand{\aeps}[0]{a^{\epsilon}}
\newcommand{\ueps}[0]{u^{\epsilon}}
\newcommand{\losseps}[0]{\mathcal{L}^{\epsilon}}
\newcommand{\Leps}[0]{L^{\epsilon}}   
\newcommand{\Keps}[0]{K^{\epsilon}}   

\newtheorem{thm}{Theorem}[section]
\newtheorem{cor}[thm]{Corollary}
\newtheorem{lem}[thm]{Lemma}

\def\ceil#1{\lceil #1 \rceil}

\def\1{\bm{1}}

\DeclareMathAlphabet{\mathsfit}{\encodingdefault}{\sfdefault}{m}{sl}
\SetMathAlphabet{\mathsfit}{bold}{\encodingdefault}{\sfdefault}{bx}{n}

\usepackage{hyperref}
\usepackage{comment}
\usepackage{graphicx}
\usepackage{url}

\title{Physics informed neural networks for elliptic equations with  \\ oscillatory differential operators}

\author{\name Arnav Gangal \email agangal00@gmail.com \\
      \addr Department of Mathematics\\
      University of California, Los Angeles 
      \AND
      \name Luis Kim \email luisyookim02@gmail.com \\
      \addr Department of Mathematics \\
      The University of Texas at Austin
      \AND
      \name Sean P.\ Carney \email spcarney@math.ucla.edu \\
      \addr Department of Mathematics \\
      University of California, Los Angeles \\
      }

\def\month{10}  
\def\year{2023} 
\def\openreview{\url{https://openreview.net/forum?id=QfyVqvpg7u}} 

\begin{document}

\maketitle

\begin{abstract}
Physics informed neural network (PINN) based solution methods 
for differential equations have recently shown success in a
variety of scientific computing applications.
Several authors have reported difficulties, 
however, when using PINNs to solve equations with multiscale
features. The objective of the present work is to illustrate 
and explain the difficulty of using standard PINNs for the particular
case of 
divergence-form elliptic partial differential equations (PDEs) with oscillatory 
coefficients present in the differential operator.
We show that if the coefficient in the elliptic operator 
$a^{\epsilon}(x)$ is of the form $a(x/\epsilon)$ for 
a 1-periodic coercive function $a(\cdot)$, 
then the Frobenius norm of the neural tangent kernel (NTK) matrix 
associated to the loss function 
grows as $1/\epsilon^2$. 
%
%
%
%
%
%
%
%
%
This implies that as the separation 
of scales in the problem increases, training the neural network
with gradient descent based methods
to achieve an accurate approximation of the solution to the 
PDE becomes increasingly difficult. 
Numerical examples illustrate the 
stiffness of the optimization problem.

\end{abstract}

\section{Introduction}
\label{sec:intro}
Recent developments in deep learning have shown great promise for advancing computational and 
applied mathematics \citep{weinan2021dawning}. 
Physics informed neural networks (PINNs) have recently emerged as a popular method 
for scientific computation.
Building off earlier work \citet{lagaris1998artificial,psichogios1992hybrid},
PINNs were introduced in \citet{raissi2019physics} and seek to approximate the true solution of
a differential equation by a neural network $u(x; \theta)$ parameterized 
by weights and biases $\theta$. 
Consider the following general partial differential equation (PDE) defined on some 
$\Omega \subset \mathbb{R}^n$ by a differential operator $\mathcal{N}$
\begin{align}
\mathcal{N}[u](x) = f(x)&, \qquad x \in \Omega \nonumber \\
u(x) = g(x)&, \qquad x \in \partial \Omega, \label{eq:general_pde}
\end{align}
where $x$ and $\Omega$ represent some general space-time coordinate and domain,
respectively. Here and throughout this paper, the boundary conditions are taken to be of Dirichlet
type. By the uniform approximation theorem \citep{hornik1989multilayer,leshno1993multilayer}, given
sufficient data, 
a neural network can uniformly approximate classical smooth solutions 
to \eqref{eq:general_pde}
whenever they exist. 

In practice, the network parameters $\theta$ are determined by
minimizing a loss function $\mathcal{L}$ that enforces \eqref{eq:general_pde} to 
hold for a set of $N_c$ collocation and $N_b$ boundary points, 
$\{x_i\}_{i=1}^{N_c} \in \Omega$ and $\{s_i\}_{i=1}^{N_b} \in \partial\Omega$, respectively.
The loss is
\begin{equation}\label{eq:gen_pinn_loss}
\mathcal{L}(\theta) = \frac{1}{N_c} \sum_{i=1}^{N_c} \frac12 \Big|\mathcal{N}[u](x_i;\theta)-f(x_i)\Big|^2
+  \frac{\lambda}{N_b}\sum_{i=1}^{N_b} \frac12 \Big| u(s_i; \theta) - g(s_i) \Big|^2 ,
\end{equation}
where $\lambda$ is a tunable parameter that weighs the relative importance of the 
boundary conditions. 
We refer to \citet{karniadakis2021physics} for a review of this methodology applied in a wide range of 
contexts in scientific computing.

Despite the success of PINNs in a wide variety of applications, numerous 
authors have reported difficulties applying the technique
to problems with multiscale features. Some examples include a scalar, 
nonlinear hyperbolic equation from a model of two-phase immiscible 
fluid flow in porous media \citep{fuks2020limitations}, which can 
support shock waves, and 
systems of ordinary differential equations governing chemical kinetics 
\citep{ji2021stiff}, which exhibit stiff dynamics that evolve over a 
wide range of time scales. 
Difficulties have also been reported for the Helmholtz equation \citep{wang2021understanding}, 
as well as linear hyperbolic problems, for example 
for example the one-dimensional advection equation 
\citep{krishnapriyan2021characterizing} and 
the wave equation \citep{wang2021eigenvector}. 

The focus of the current work is to illustrate and explain the difficulty of using standard PINN
solution methods for linear elliptic boundary value problems (BVPs) of the form
\begin{align}
-\nabla \cdot \big( a^{\epsilon}(x) \nabla u^{\epsilon}(x) \big) &= f(x), \qquad x \in \Omega \nonumber \\
\ueps(x) &= g(x), \qquad x \in \partial \Omega, \label{eq:darcys_law}
\end{align}
where the 
coefficient tensor $\aeps: \mathbb{R}^n \to \mathbb{R}^n$ is uniformly bounded and coercive in $\epsilon$ and
assumed to consist of entries that contain frequencies on 
the order of $\epsilon^{-1}$ for $ 0 < \epsilon \ll 1$. The measure of the 
domain $\Omega \subset \mathbb{R}^n$ is assumed to be $\bigoh{1}$, and hence, \eqref{eq:darcys_law} 
is a multiscale problem that models, for example, steady-state heat conduction 
in a composite material or porous media flow governed by Darcy's law.

Theoretical analysis and numerical experiments presented below illustrate that,
whenever standard PINN architectures are used in 
conjunction with gradient descent based training for the multiscale problem 
\eqref{eq:darcys_law}, the resulting optimization problem becomes increasingly difficult as the scale 
separation in the BVP increases, i.e.\ as $\epsilon$ vanishes. After motivating the 
present study in Section \ref{sec:motivation}, 
we show in
Section \ref{sec:ntk} that the  
neural tangent kernel matrix associated with the PINN approximation to 
\eqref{eq:darcys_law} has a Frobenius norm that becomes unbounded 
as $\epsilon \downarrow 0$. 
Numerical examples in Section \ref{sec:numerical_results} illustrate that during training, the 
ordinary differential equation that governs the evolution of the BVP residuals indeed becomes
increasingly stiff as $\epsilon \downarrow 0$, translating to poor PINN performance for 
problems with a large separation of scales.

\section{Motivation} 
\label{sec:motivation}
The motivation for attempting to use physics informed neural network solutions 
to the oscillatory problem \eqref{eq:darcys_law} is to investigate whether
a connection can be established between asymptotic homogenization theory and the 
so-called ``frequency principle'' in deep learning.  

Recall from mathematical homogenization theory \citep{bensoussan2011asymptotic} 
that, under suitable conditions, the solution to \eqref{eq:darcys_law} is well 
approximated as $\epsilon\downarrow 0$ by the solution to a homogenized equation 
\begin{align*}
-\nabla \cdot \big( \overline{a}(x) \nabla \overline{u}(x) \big) &= f(x), \qquad x \in \Omega \nonumber \\
\overline{u}(x) &= g(x), \qquad x \in \partial \Omega. \label{eq:hom_eqn}
\end{align*}
Both the homogenized coefficients $\overline{a}$ and the solution $\overline{u}$
do not contain $\epsilon$-scale oscillations; the latter approximates the large-scale, 
low-frequency features of the oscillatory function $\ueps$. 

Additionally, 
when neural networks learn a target function, they are known
to learn the low frequency components more rapidly than the large 
frequencies \citep{rahaman19,xu2020}. This ``frequency
principle'' was shown to hold for gradient descent training in \citet{luo2021,markidis2021old}; 
empirically the result can be observed for PINNs applied to 
relatively simple problems \citep{wang2021eigenvector}, even when more widely used optimizers are
used for training, e.g.\ Adam \citep{kingma2014adam}. 
Consider as a brief representative example the PINN solution to the one-dimensional 
Poisson BVP
\begin{equation}\label{eq:1d_poisson}
-\frac{d^2}{dx^2} u (x) = \sin(x) + \sin(5x) + \sin(15x) + \sin(55x) =: f(x) 
\end{equation}
for $x \in (-\pi,\pi)$ with homogeneous Dirichlet boundary conditions 
$u(-\pi) = u(\pi) = 0$. 
For collocation points $\{x_i\}_{i=1}^{N_c} \in (-\pi,\pi)$ the loss function becomes 
\begin{equation}\label{eq:1d_poisson_loss}
\mathcal{L}(\theta) = \frac{1}{N_c} \sum_{i=1}^{N_c} \frac12 
\Big| \frac{d^2}{dx^2} u(x_i; \theta) + f(x_i) \Big|^2 + 
\frac14 \lambda \left( \big|u(-\pi;\theta)\big|^2 + \big|u(\pi;\theta)\big|^2\right).
\end{equation}
Figure \ref{fig:freq_principle_1d_poisson} shows the evolution (as a function of training
iteration) of the complex modulus of 
\begin{equation}\label{eq:fourier_error_1d_poisson}
\widehat{e}_k = (\widehat{u}_{\rm true})_k - (\widehat{u}_{\rm NN})_k 
\end{equation}
where $u_{\rm true}$ is the true solution to the BVP, $u_{\rm NN}$ is the 
neural network approximation, $\widehat{\cdot}$ denotes the discrete 
Fourier transform, and the index $k$ denotes the frequency of the coefficient. A rolling average is used to make the trajectories more legible.  
The PINN solution is computed with a fully connected neural network with four hidden layers 
of sixty nodes each, $N_c =512$ equispaced collocation points, and $\lambda=80$; 
see Appendix \ref{sec:appendix0} for a complete description of the training process that generates the final neural network output. 
One can clearly observe the frequency principle in this simple example; the low frequencies of the target BVP solution 
are more rapidly learned than the high frequencies. 
See also \citet{markidis2021old} for another example in a two-dimensional Poisson problem, as well as \citet{wang2021eigenvector} for 
mathematical analysis of this phenomenon for Poisson equations. 
\begin{figure}[h]
    \centering
 \includegraphics[width=0.75\textwidth]{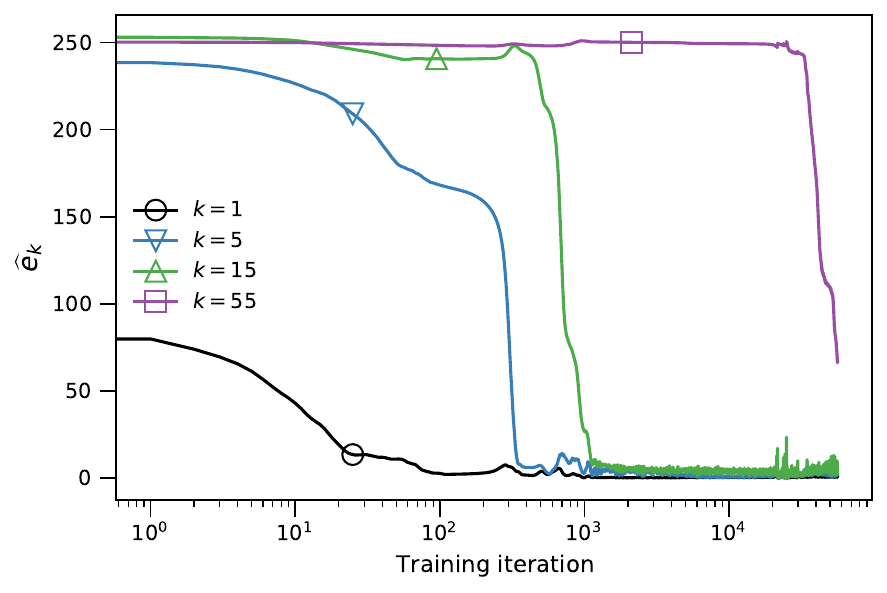}
\caption{A simple illustration of the frequency principle: evolution 
as a function of training iteration
of the magnitude of the discrete Fourier transform of the error \eqref{eq:fourier_error_1d_poisson}
between the true solution to the BVP \eqref{eq:1d_poisson} and the neural network approximation. 
} 
\label{fig:freq_principle_1d_poisson}
\end{figure}

Given some evidence that PINNs can learn low frequency features of 
PDE solutions, it is reasonable to ask if the homogenized solution 
`naturally' arises when trying to learn solutions to the multiscale 
problem \eqref{eq:darcys_law}, and, if not, for what reason? The homogenized
solution itself is of course useful in a variety of applications; it 
additionally could be used to construct coarse grid solutions for a multigrid solver
of the full, oscillatory problem as in \citet{engquist1997convergence}. The
natural smoothing properties of standard iterative solvers, e.g.\ the method of 
Gauss-Seidel or successive over-relaxation \citep{briggs2000multigrid}, would then rapidly reduce the error
of high-frequency solution components that are not easily approximated by the 
neural networks; such a `hybrid' multigrid strategy was explored in \citet{markidis2021old} 
for Poisson equations.  

Recent theoretical results in \citet{shin2020} guarantee that 
as the number of collocation $N_c$ and boundary $N_b$ points tend to infinity, the 
sequence of minimizers to a ``H\"{o}lder regularized'' version of the loss function
 \eqref{eq:gen_pinn_loss} will converge to classical solutions of elliptic and 
parabolic PDEs (whenever they exist) in the limit of infinite learning data. 
 Nevertheless, the authors in both \citet{han2022neural} and \citet{leung2021nh}
reported difficulties training neural networks to achieve such minimizers of the 
loss function associated to the multiscale problem \eqref{eq:darcys_law}. 
The purpose of the current work is to characterize the reason 
why optimization is difficult with both theory and numerical examples. 

The theoretical results presented below describe the neural tangent kernel matrix of the PINN
solution to the BVP \eqref{eq:darcys_law}; this matrix appears in the ordinary 
differential equation that, under gradient descent training dynamics, governs
the evolution of the PDE and boundary residuals that define the PINN loss function. 
Under a few technical assumptions, we show that the matrix Frobenius norm becomes 
unbounded as $\epsilon \downarrow 0$. This result, along with numerical evidence
presented in Section \ref{sec:numerical_results}, suggests that as that as the scale separation in 
the differential operator in \eqref{eq:darcys_law} increases, the optimization problem that
determines the neural network approximation to the PDE solution becomes increasingly stiff.

\section{Neural tangent kernel matrix theory} 
\label{sec:ntk}
\subsection{Theoretical results} \label{subsec:main_theory}
We now describe the neural tangent kernel (NTK) matrix for the physics informed neural network 
approximation to the multiscale elliptic equation \eqref{eq:darcys_law} before showing 
that 
it can grow arbitrarily large in norm as $\epsilon \downarrow 0$. See 
\citet{jacot2018} for the original development of the NTK theory for least-squares
regression problems, as well as \citet{wang2021eigenvector} and \citet{wang2022_JCP} 
for an extension of the theory to PINNs. 

For the presentation below it will be useful to define the linear multiscale differential 
 operator
\begin{equation*}\label{eq:Leps}
\Leps\varphi := -\nabla \cdot \big( \aeps \nabla \varphi \big),
\end{equation*}
where $\varphi$ is some suitably regular function.
The multiscale elliptic boundary value problem \eqref{eq:darcys_law} then becomes 
\begin{align}
\Leps\ueps &= f  \qquad \Omega  \nonumber \\
\ueps &= g  \qquad  \partial \Omega.  \label{eq:concise_darcy}
\end{align}
For simplicity, we assume throughout that the entries of the coefficient tensor 
$\aeps$ are at least once continuously differentiable, and hence bounded on 
any compact domain $\Omega$.

Given some neural network $u(x;\theta)$ parameterized by $N_p$ weights and biases 
$\theta$, as well as collocation $\{x_i\}_{i=1}^{N_c} \in \Omega$ and boundary points
$\{s_i\}_{i=1}^{N_b} \in \partial \Omega$, the loss function associated 
to \eqref{eq:concise_darcy} 
becomes 
\begin{equation}\label{eq:loss_eps}
\losseps(\theta):= \frac{1}{N_c} \sum_{i=1}^{N_c} \frac12 
\Big| \Leps u(x_i;\theta) - f(x_i)\Big|^2 + \frac{\lambda}{N_b} \sum_{i=1}^{N_b} \frac12 \Big| u(s_i; \theta)-g(s_i) \Big|^2. 
\end{equation}
If the residual values
\begin{equation}\label{eq:residual_pde}
r_{\rm pde}(x_i; \theta) := \Leps u(x_i; \theta)-f(x_i), \qquad i = 1,\ldots, N_c
\end{equation}
and 
\begin{equation}\label{eq:residual_b}
r_{\rm b}(s_i; \theta) := u(s_i; \theta) - g(s_i), \qquad i=1,\ldots , N_b
\end{equation}
are grouped together into a single vector $y(\theta)$, and if
the parameters $\theta$ evolve according to the gradient flow
\begin{equation}\label{eq:grad_flow}
\frac{d\theta}{dt} = - \nabla_{\theta} \losseps ,
\end{equation}
then $y(\theta(t))$ will evolve according to an
initial value problem of the form 
\begin{equation}\label{eq:y_evolution}
\frac{d}{dt} y(t)  = -\Keps(y(t))\,  y(t),  
\end{equation}
where $\Keps$ is called the neural tangent kernel matrix. 
Note here that both $\theta$ and $y$ implicitly depend on 
$\epsilon$; however, the dependence
is not explicitly marked.  

We next derive \eqref{eq:y_evolution}, noting again that 
similar results can be found in \citet{wang2021eigenvector} and \citet{wang2022_JCP}.

\begin{lem} \label{lem:ntk}               
For neural network parameters $\theta \in \mathbb{R}^{N_p}$,
let $y(\theta) \in \mathbb{R}^{N_c + N_b}$ be the vector of residual values 
$$
y(\theta) = \begin{pmatrix}
r_{\rm pde}(x_1; \theta), \ldots, r_{\rm pde}(x_{N_c}; \theta),
r_{\rm b}(s_1; \theta), \ldots, r_{\rm b}(s_{N_b}; \theta) 
\end{pmatrix}^T
$$
where the entries $r_{\rm pde}$ and $r_{\rm b}$ are defined by 
\eqref{eq:residual_pde} and \eqref{eq:residual_b}. Suppose that 
the parameters $\theta$ evolve from some initial value $\theta_0$
according to the gradient flow \eqref{eq:grad_flow}. Then $y(\theta(t))$
evolves from the initial condition $y(\theta_0)$ according to
\begin{equation}\label{eq:ntk_ivp}
\frac{d}{dt} y(t)  = -\Keps(y(t))\,  y(t),  
\end{equation}
where the explicit dependence on $\theta$ in \eqref{eq:ntk_ivp} is 
dropped for convenience. 
The neural tangent kernel matrix is given by 
\begin{equation} \label{eq:Keps}
\Keps(t) = \begin{pmatrix} \Keps_{uu}(t) & \Keps_{ub}(t) \\ \Keps_{bu}(t) & K_{bb}(t) \end{pmatrix} 
\in \mathbb{R}^{(N_c+N_b)\times (N_c+ N_b) },
\end{equation}
where the subblocks have entries
\begin{equation}
    (\Keps_{uu})_{ij}(t) = \frac{1}{N_c} \sum_{l=1}^{N_p} 
\pderiv{}{\theta_l}\Leps u(x_i; \theta(t)) 
\pderiv{}{\theta_l}\Leps u(x_j; \theta(t)),
\qquad 1\le i,j \le N_c,
\nonumber
\end{equation}
\begin{equation}
    (K_{bb})_{ij}(t) =  \frac{\lambda}{N_b} \sum_{l=1}^{N_p} 
\pderiv{}{\theta_l}u(s_i; \theta(t)) 
\pderiv{}{\theta_l}u(s_j; \theta(t)),
\qquad 1\le i,j \le N_b,
\nonumber
\end{equation}
and 
\begin{equation}
    (\Keps_{ub})_{ij}(t) = \frac{\lambda}{N_b} \sum_{l=1}^{N_p} 
\pderiv{}{\theta_l}\Leps u(x_i; \theta(t)) 
\pderiv{}{\theta_l}u(s_j; \theta(t)),
\qquad 1\le i \le N_c, \, \, 1 \le j \le N_b, 
\nonumber
\end{equation}
while $\Keps_{bu}$ is simply the scaled transpose of $\Keps_{ub}$: 
$$
(\Keps_{bu})_{ij} = \frac{N_b}{\lambda} \frac{1}{N_c} \, \, (\Keps_{ub})_{ji},
\qquad 1\le i \le N_c, \, \, 1 \le j \le N_b. 
$$
\end{lem}
The proof is a computation that follows from the chain rule from differential calculus;
it can be found in Appendix \ref{sec:appendix1}. We note that while the $\Keps_{uu}$ and $\Keps_{bb}$ subblocks
of the NTK matrix are symmetric, overall, $\Keps$ is not symmetric unless $N_b = \lambda N_c$. 

Next we show that the NTK matrix \eqref{eq:Keps} associated to the multiscale elliptic 
problem--in particular the $\Keps_{uu}$, $\Keps_{ub}$ and $\Keps_{bu}$ subblocks--can become 
arbitrarily large for vanishing $\epsilon$. In contrast, because the boundary condition 
$g$ in \eqref{eq:concise_darcy} is indepedent of $\epsilon$, there is no explicit 
$\epsilon$-dependence in the $K_{bb}$ subblock.

For simplicity we assume the problem is one-dimensional 
($n=1$) and that the multiscale coefficient $\aeps(x) = a(x/\epsilon)$ is $\epsilon$-periodic.
After showing the desired result  
for $n=1$, we briefly sketch below an additional proof for the higher dimension case $n>1$ under 
additional technical assumptions.

In dimension $n=1$, \eqref{eq:concise_darcy} becomes
\begin{align}
- \frac{d}{dx}\left(a(x/\epsilon) \frac{d}{dx}\, \ueps(x) \right) &= f(x),  \qquad x \in (a,b)  \nonumber \\
\ueps(a) = u_a, \qquad \ueps(b) &= u_b.   \label{eq:1d_darcy} 
\end{align}
 The loss function associated to \eqref{eq:1d_darcy} is equation \eqref{eq:loss_eps}
with $N_b=2$ and elliptic operator $\Leps$ given by
\begin{equation} \label{eq:Leps_expanded_1d}
\Leps u(x;\theta) = -\Big( a(x/\epsilon) \frac{d^2}{dx^2}\, u(x;\theta) + \frac{1}{\epsilon} a'(x/\epsilon) \frac{d}{dx} u(x;\theta) \Big), 
\end{equation}
where it is assumed the oscillatory function $a$ is differentiable. 
The $1/\epsilon$ factor in the second term will be the source of the divergent 
behavior of the NTK as $\epsilon \downarrow 0$. 

We first show the result for a neural 
network with one hidden layer before extending 
it to more general fully connected networks.


\begin{thm}\label{thm:main_thm}             
Let $u:\mathbb{R} \to \mathbb{R}$ be a neural network with one hidden layer of width $d$ and 
smooth activation function $\sigma$, so that
\begin{equation}\label{eq:1layer_nn}
u(x;\theta) = \sum_{k=1}^d W_k^{(1)} \sigma \Big(W_{k}^{(0)} x + b_k^{(0)} \Big) + b^{(1)} , 
\end{equation}
where $b^{(1)} \in \mathbb{R}$, 
$W^{(1)} \in \mathbb{R}^{1\times d}$, and
 $W^{(0)}, b^{(0)} \in \mathbb{R}^{d\times 1}$, 
and assume that $\forall x\in \mathbb{R}$
\begin{equation}\label{eq:sigma_assumption}
0 < \sigma'(x) <  \infty. 
\end{equation} 
Let $a(y)$ be a one-periodic, non-constant $C^1$ function that is bounded and coercive, so that $\forall y \in \mathbb{R}$, 
there exist some $\lambda, \Lambda \in \mathbb{R}$ such that 
$$
0< \lambda  \le a(y) \le \Lambda,
$$ 
and let $\aeps(x) = a(x/\epsilon)$. 
For a fixed number of collocation points $N_c$, let $\Keps$ be the neural tangent
kernel matrix \eqref{eq:Keps} associated to the loss function \eqref{eq:loss_eps} of the 
boundary value problem \eqref{eq:1d_darcy}, and 
suppose the network parameters $\theta$ evolve according to the gradient flow \eqref{eq:grad_flow}
for $0 \le t \le T$ for some $T>0$. 

Suppose also that the parameters $\theta$ evolving according to \eqref{eq:grad_flow} are bounded uniformly
in both $t$ and $\epsilon$:
\begin{equation}\label{eq:assumption2}
\sup_{0 < \epsilon < \epsilon_0 } \sup_{0\le t \le T} \norm{\theta(t)}_{\infty} < C
\end{equation}
for some $\epsilon_0 > 0$. 
Finally, let $\Upsilon$ be the set of $t \in[0,T]$ for
which there exists at least one entry of $W^{(0)}$ that
is asymptotically larger than $\epsilon$; that is, $t \in \Upsilon$ if
and only if 
$$
\lim_{\epsilon \downarrow 0 } \epsilon/W^{(0)}_{l}(t) = 0
$$
for some $1 \le l \le d$. 
Then $\forall t \in \Upsilon$
\begin{equation*}\label{eq:final_result}
\lim_{\epsilon \downarrow 0} \norm{\Keps(t)}_F = \infty 
\end{equation*} 
where $\norm{\cdot}_F$ is the Frobenius norm.
\end{thm}

Before proving the result, we briefly remark on the theorem assumptions. 
First note that the number of collocation points $N_c$ is assumed to be fixed. 
Since we will show that the entries of the NTK matrix are proportional to 
$$\frac{1}{N_c } \, \frac{1}{\epsilon^2},$$ the singular nature of the evolution of the residuals
\eqref{eq:y_evolution} could in principle be prevented by taking $N_c \sim \epsilon^{-2}$. However, 
this can make the computational cost of gradient descent based optimization prohibitively 
expensive whenever the problem scale separation is sufficiently large. 

Next, note the assumption \eqref{eq:sigma_assumption} on $d \sigma/dx$ holds for sigmoidal-type 
activation functions commonly used for PINNs \citep{lu2021deepxde}, such as the hyberbolic tangent and logistic 
functions. Although it does not hold for the ReLU function $\sigma(x) = \max(0,x)$, these 
are of course not suitable for PINNs and other NN based solution methods 
that are based on the strong formulation of PDEs with 
second order (or higher) differential operators.

Regarding the set $\Upsilon \subseteq [0,T]$, note that if $t \notin \Upsilon$, then each 
entry of $W^{(0)}$ will converge 
to $0$ in the limit $\epsilon \downarrow 0$; in this scenario the neural network 
\eqref{eq:1layer_nn} would limit to a constant function and hence not be a suitable 
solution to \eqref{eq:1d_darcy} in general. Finally, 
if the uniform bound \eqref{eq:assumption2} on the network 
parameters $\theta$ did not hold, then an unstable, divergent training process could 
result \citep{wang2022_JCP}; the bound helps prevent the so-called ``vanishing gradient'' 
problem that can occur when components 
of $\nabla_{\theta} \losseps$ vanish, preventing descent in that direction. 


\begin{proof}[Proof of Theorem \ref{thm:main_thm}]              
The proof, like that of Lemma \ref{lem:ntk}, is a consequence of direct computation; 
note that it is sufficient for a single entry of a matrix to diverge to ensure
the Frobenius norm diverges. In particular we show that entries in the $\Keps_{uu}$
subblock of the full NTK matrix $\Keps$ scale as $1/(N_c\,\epsilon^2)$. 

Indeed, recall from Lemma \ref{lem:ntk} that 
\begin{equation} \label{eq:Kuu_entry}
(\Keps_{uu})_{\alpha\alpha}(t) = \frac{1}{N_c} \sum_{\gamma=1}^{N_p} 
\Big(\pderiv{}{\theta_{\gamma}} \Leps u(x_{\alpha}; \theta(t)) \Big)^2  
\end{equation}
for any $1\le \alpha\le N_c$. 
Since \eqref{eq:Kuu_entry} is a sum of squares, if
just one entry in the sum blows up as $\epsilon \downarrow 0$, 
then the entire sum of course will as well. 
Below we drop the explicit dependence of the parameters $\theta$ on $t$ for
notational convenience. 

First, let $g(x;\theta)$ be the $\mathbb{R}^d$ valued function whose $k$-th component
equals 
$$
g_k(x;\theta) = \sigma \Big(W_{k}^{(0)} x + b_k^{(0)} \Big),
$$
so that 
$$
u(x;\theta) = \sum_{k=1}^d W_k^{(1)} g_k(x;\theta)) + b^{(1)}.
$$ 
By linearity, 
$$
\pderiv{}{W^{(1)}_{\gamma}} \Leps u(x;\theta)  = \Leps g_{\gamma}(x;\theta)
$$
for any $1 \le \gamma  \le d$.
So, to prove the desired result, it suffices to show that 
\begin{equation} \label{eq:show_blowsup}
\frac{1}{N_c} \big[\Leps g_{\gamma}(x_{\alpha};\theta) \big]^2  \to \infty
\end{equation}
as $\epsilon \downarrow 0$. Using \eqref{eq:Leps_expanded_1d}, the left-hand side 
of \eqref{eq:show_blowsup} equals 
\begin{align}
\frac{1}{N_c} \Big[ 
\Big( a(x_{\alpha}/\epsilon) \frac{d^2}{dx^2} &g_{\gamma}(x_{\alpha}; \theta)\Big)^2 + 
\frac{2}{\epsilon} \Big(a(x_{\alpha}/\epsilon) \frac{d^2}{dx^2} g_{\gamma}(x_{\alpha}; \theta)
\, a'(x_{\alpha}/\epsilon)  \frac{d}{dx} g_{\gamma}(x_{\alpha}; \theta)\Big)\nonumber  \\
& +  \frac{1}{\epsilon^2} \Big(a'(x_{\alpha}/\epsilon) \frac{d}{dx} g_{\gamma}(x_{\alpha};\theta)\Big)^2\Big].
\label{eq:final_blowup}
\end{align}
Since $a$ is a periodic $C^1$ function, both $a(x_{\alpha}/\epsilon)$ and $a'(x_{\alpha}/\epsilon)$ are bounded 
independent of $x_{\alpha}/\epsilon \in\mathbb{R}$. The derivatives 
of $g_{\gamma}$ are
\begin{equation}\label{eq:g_1stderiv}
\frac{d}{dx}\, g_{\gamma}(x;\theta) = 
\sigma' \Big(W_{\gamma}^{(0)} x + b_{\gamma}^{(0)} \Big) W_{\gamma}^{(0)}
\end{equation}
and
\begin{equation}\label{eq:g_2ndderiv}
\frac{d^2}{dx^2}\, g_{\gamma}(x;\theta) = 
\sigma'' \Big(W_{\gamma}^{(0)} x + b_{\gamma}^{(0)} \Big) \big(W_{\gamma}^{(0)}\big)^2. 
\end{equation}
Since the activation function is smooth and the network parameters are uniformly bounded in $\epsilon$ and $t$, 
both \eqref{eq:g_1stderiv} and \eqref{eq:g_2ndderiv} are also bounded for all $x_{\alpha} \in [a,b]$. 
Consequently, \eqref{eq:final_blowup} is dominated by the $1/\epsilon^2$ term for $\epsilon$
vanishingly small. The positivity assumption \eqref{eq:sigma_assumption} on $d\sigma/dx$
and the uniform boundedness of the network parameters imply also that the (absolute value
of the) first derivative
\eqref{eq:g_1stderiv} is bounded below by a constant independent of $\epsilon$; 
this implies that for any $t \in \Upsilon$
\begin{equation*}
\lim_{\epsilon \downarrow 0} \frac{1}{\epsilon} \left| \frac{d}{dx}\, g_l(x;\theta)\right| = \infty
\end{equation*}
holds for at least one $l \in \{1,\ldots d\}$. Take $\gamma = l$, and let $\epsilon$ vanish 
monotonically to zero in such a way that $a'(x_{\alpha}/\epsilon) \ne 0$ for any $\epsilon$ (such a sequence 
exists since $a$ is periodic and non-constant). Then \eqref{eq:final_blowup} indeed limits to positive 
infinity, giving the desired result.

\end{proof}               

Without giving a formal proof, we now briefly describe sufficient conditions under which Theorem \ref{thm:main_thm}
could be extended to the more general case $n >1$ (assuming one retains the positivity assumption \eqref{eq:sigma_assumption}
on $d\sigma/dx$ and the uniform boundedness assumption \eqref{eq:assumption2} on the network parameters). 
A neural network with one hidden layer of width $d$ would be 
$$
u(x;\theta) = \sum_{k=1}^d W_k^{(1)} \sigma\Big( \sum_{l=1}^n W^{(0)}_{kl} x_l + b^{(0)} \Big) + b^{(1)} ;
$$ 
here $x_l$ denotes the $l$-th component of the function input $x \in \mathbb{R}^n$.
Following the same argument just given, define 
$$
g_{k}(x; \theta) = \sigma\Big( \sum_{l=1}^n W^{(0)}_{kl} x_l + b^{(0)} \Big),
$$
and assume that there exists some $1 \le \mu \le d$ and $1 \le \nu \le n$ such that 
$W^{(0)}_{\mu \nu}$ is asymptotically larger than $\epsilon$ (as before, if no such 
entries exist, then $W^{(0)}$ limits to the zero matrix as $\epsilon \downarrow 0$,
and hence $u(x;\theta)$ limits to a constant function). 
The Frobenius norm of the PINN's neural tangent kernel  
matrix will blow-up if $[\Leps g_{\mu}(x;\theta)]^2$ blows-up at
some collocation point $x$. This can occur if the entries for which $\nabla \cdot a$ 
are nonzero coincide with the entries where $W^{(0)}$ is larger than $\bigoh{\epsilon}$; 
more precisely, if there is a sequence $\epsilon\downarrow 0$ such that    
$$
\sum_{i=1}^n \pderiv{a_{i\nu}}{x_i} \Big( \frac{x}{\epsilon}\Big) 
\sigma'\Big( \sum_{l=1}^n W_{\mu l}^{(0)} x_{l} + b_{\mu}^{(0)} \Big) W_{\mu \nu}^{(0)} \ne 0
$$
for any collocation point $x$, then the main result will generalize.

Returning to the case of one-dimension ($n=1$), we now extend Theorem \ref{thm:main_thm}
to more general fully connected neural networks under an additional 
assumption on the network parameter's asymptotic behavior.
%
\begin{thm}\label{thm:second_thm}             
Let $u: \mathbb{R} \to \mathbb{R}$ be a fully connected neural network 
with $\Lambda$ hidden layers of widths $d_1, d_2, \ldots, d_{\Lambda}$, 
so that 
\begin{align}
u^{(0)}(x) &= x,    \nonumber \\
u^{(l)}(x) &= \sigma\big( W^{(l)} u^{(l-1)}(x) + b^{(l)}\big), \qquad 1 \le l \le \Lambda, \nonumber\\
u(x; \theta) &= W^{(\Lambda + 1)} \cdot u^{(\Lambda)}(x) + b^{(\Lambda+1)}.  \label{eq:full_nn}
\end{align}
If all of the assumptions from Theorem \ref{thm:main_thm} are retained, but $\Upsilon$ 
is now defined to be the set of $t \in [0,T]$ such that the magnitude of every entry of each matrix
$W^{(l)}$, $1 \le l \le \Lambda$, 
are asymptotically larger than $\epsilon^{1/\Lambda}$, 
then $\forall t \in \Upsilon$
$$
\lim_{\epsilon \downarrow 0} \norm{\Keps(t)}_{F} = \infty, 
$$
where $\Keps$ is the NTK matrix
associated to the loss function \eqref{eq:loss_eps} of the 
boundary value problem \eqref{eq:1d_darcy}.
\end{thm}
The proof proceeds nearly identically to that of Theorem \ref{thm:main_thm}, and hence
it is in Appendix \ref{sec:appendix2}. 

An immediate corollary of Theorem \ref{thm:second_thm} is that the spectral radius of $\Keps_{uu}$ 
must also blow up as $\epsilon \downarrow 0$; in contrast, because there is no explicit 
$\epsilon$-dependence in the $K_{bb}$ subblock, one expects its spectral radius to be bounded
independently of $\epsilon$, which is indeed numerically observed in section \ref{subsec:ntk_scaling}.
\begin{cor} \label{cor:spectral_radius}

Let $\rho(\Keps_{uu}(t))$ denote the spectral radius of $\Keps_{uu}(t)$. 
Under the same assumptions as Theorem \ref{thm:second_thm}, we have for any $t \in \Upsilon$ that
$$
\lim_{\epsilon\downarrow 0} |\rho(\Keps_{uu}(t)) | = \infty .
$$
\end{cor}
See Appendix \ref{sec:appendix3} for the simple proof.

\subsection{Discussion}
\label{subsec:discussion}
Fundamentally, using PINNs to numerically solve PDEs involves optimizing a multiobjective loss functional, 
for example of the form 
$$
\mathcal{L}(\theta)
= \mathcal{L}_{\rm PDE}(\theta) + \mathcal{L}_{\rm BC}(\theta)
$$
for boundary values problems. 
The present work builds on the previous studies \cite{wang2021understanding}, \cite{wang2021eigenvector} 
and \cite{wang2022_JCP}, where certain PINN failure 
modes were identified for multiscale problems that, broadly speaking, involve inbalances between $\mathcal{L}_{\rm PDE}$ and $\mathcal{L}_{\rm BC}$.

These previous analyses focused on unsatisfactory PINN training in the particular case of Poisson-type PDEs, i.e.\ 
where the differential operator is the Laplacian (or, in the one-dimensional case, the second derivative). 
The multiscale character of the problems originated from oscillatory forcing functions; an illustrative example is the boundary value
problem
\begin{equation} \label{eq:illust_examp}
-\frac{d^2}{dx^2} \ueps(x) = \sin(x) + (1/\epsilon) \sin(x/\epsilon)
\end{equation}
on $[0,1]$ along with appropriate boundary conditions (see e.g.\ Eq.~(2.6) in \cite{wang2021eigenvector}).
The present work focuses on elliptic problems that feature a more general class of divergence-form differential operators
with oscillatory coefficients $\aeps(x)$, as in Eq.~\eqref{eq:darcys_law}. The differences between the two cases can 
be considerable. 

For example, from \cite{wang2021understanding} it is known that PINNs can fail to train when the magnitudes of the 
gradients of the two different loss components ($\mathcal{L}_{\rm PDE}$ and $\mathcal{L}_{\rm BC}$) with respect to 
the network parameters $\theta$ are imbalanced, i.e.\ when either $|\nabla_{\theta} \mathcal{L}_{\rm PDE}| \gg |\nabla_{\theta} \mathcal{L}_{\rm BC}|$ or 
$|\nabla_{\theta} \mathcal{L}_{\rm PDE}| \ll |\nabla_{\theta} \mathcal{L}_{\rm BC}|$. For multiscale Poisson equations 
such as \eqref{eq:illust_examp}, the $1/\epsilon$ term in the right-hand side forcing results in $|\nabla_{\theta} \mathcal{L}_{\rm PDE}|$ scaling like $1/\epsilon$. 
In contrast, for Darcy-type problems such as \eqref{eq:darcys_law} (or, in one-dimension, \eqref{eq:1d_darcy}), 
$|\nabla_{\theta} \mathcal{L}_{\rm PDE}|$ is more singular; it scales like $1/\epsilon^2$. In both cases, 
$|\nabla_{\theta} \mathcal{L}_{\rm BC}|$ is indepedent of $\epsilon$, so the discrepancy is greater for problems considered 
in the present work. 

An alternative PINN failure mode analyzed in \cite{wang2021eigenvector} and \cite{wang2022_JCP} is when the
the eigenvalues of the different NTK matrix subblocks differ greatly in modulus. 
The authors reported a large imbalance between the eigenvalues of the $K_{uu}$ and $K_{bb}$ subblocks  
for Poisson-style problems. However, for equations such as \eqref{eq:illust_examp}, it is clear from 
the definitions in \eqref{eq:Keps} that the $K_{uu}$ subblock is independent of $\epsilon$, since the forcing function 
$f$ (which is responsible for the equation's multiscale character)
is independent of $\theta$.

In contrast, 
the results from Section \ref{subsec:main_theory} show that while the $K_{bb}$ subblock is independent of $\epsilon$ for
Darcy-type problems, the spectral radius of the $K_{uu}$ subblock increases as $1/\epsilon^2$. Hence, for problems with 
large scale separation, the imbalance in eigenvalues is even more severe than in the Poisson case. 
%
For this reason,  
training PINNs with gradient based optimizers
to achieve an accurate approximation to solutions to multiscale equations of the form \eqref{eq:darcys_law} becomes
increasingly difficult as a function of increasing problem scale separation. 

A common strategy 
to deal with the imbalances just highlighted is to assign the different loss components 
$\mathcal{L}_{\rm PDE}$ and $\mathcal{L}_{\rm BC}$ different weights $\lambda_{\rm PDE}$ and $\lambda_{\rm BC}$; 
this was 
proposed e.g.\ in \cite{wang2021understanding}, \cite{van2022optimally} and \cite{wang2022_JCP}.
It is reasonable, for example, for the equations considered here to try and weight $\mathcal{L}_{\rm PDE}$ 
by $\lambda_{\rm PDE} = \epsilon^2$ to cancel the $1/\epsilon^2$ factor in the NTK $\Keps_{uu}$ subblock. 
Although this of course changes the magnitude of the eigenvalues of the NTK subblock, it does not change their
distribution, and hence cannot directly resolve spectral bias \cite{wang2022_JCP}, 
which may explain the lack of success using adaptive weighting strategies that was reported in both 
\cite{leung2021nh} and \cite{han2022neural}. The numerical examples in Section \ref{subsec:rescale_test} below 
illustrate that even under this rescaling, the stiffness of the gradient flow dynamics described by \eqref{eq:grad_flow} 
remains, translating to poor PINN performance.

\section{Numerical results} 
\label{sec:numerical_results}
\subsection{Scaling of the neural tangent kernel matrix}\label{subsec:ntk_scaling}
Theorems 
\ref{thm:main_thm} and \ref{thm:second_thm}  
state that the Frobenius norm of the neural tangent kernel (NTK) matrix $\Keps$ associated to the 
PINN approximation of the multiscale BVP \eqref{eq:1d_darcy} should scale inversely 
proportional to $N_c \,  \epsilon^2$ for gradient descent training with infinitesimal 
time steps.
Here we numerically reproduce this scaling for a neural network trained in a more practical 
setting, namely with the Adam optimizer \citep{kingma2014adam}. 

As in Theorems \ref{thm:main_thm} and \ref{thm:second_thm}, we focus on the $\Keps_{uu}$ 
subblock of the NTK matrix. For a neural network with one hidden layer $d$, the $ij$-th 
entry of $\Keps_{uu}$ is given by
\begin{align*}
\big(\Keps_{uu}\big)_{ij}(t) = 
\sum_{k=1}^{d} \Big( & \pderiv{}{W_k^{(1)}} \Leps u(x_i; \theta(t))\pderiv{}{W_k^{(1)}} \Leps u(x_j; \theta(t)) 
+ \pderiv{}{W_k^{(0)}} \Leps u(x_i; \theta(t))\pderiv{}{W_k^{(0)}} \Leps u(x_j; \theta(t)) \\ 
&+ \pderiv{}{b_k^{(0)}} \Leps u(x_i; \theta(t))\pderiv{}{b_k^{(0)}} \Leps u(x_j; \theta(t)) \Big)
\end{align*}
(note $\partial \Leps u(x;\theta)/\partial b^{(1)} = 0$). 
For network width $d = 50$ and $N_c = 256$ collocation points, 
Figure \ref{fig:ntk_scaling}(a) shows the Frobenius
norm $\norm{\Keps_{uu}}_F$ associated to \eqref{eq:1d_darcy} for the particular case that $[a,b]=[-\pi,\pi]$, 
$u_a = u_b = 0$, and $\aeps(x) = 1/(2.1 + 2\sin(2\pi x/\epsilon))$. The norm is computed at time $t=0$, i.e. 
at initialization, for a sequence of $\epsilon$ values between $1/10$ and $1/100$. For every $\epsilon$ value, 
 the neural network parameters are initialized from the normal distribution
with mean zero and unit variance. Figure \ref{fig:ntk_scaling}(b) shows the Frobenius norms after initializing
with the Glorot distribution \citep{glorot2010understanding} and then training for ten thousand 
iterations with the Adam optimizer at a learning rate of $\eta=10^{-5}$.
In both cases, the norm increases as $1/\epsilon^2$, consistent with the theory developed above.

\begin{figure}[h]       
    \centering
(a)
 \includegraphics[width=0.45\textwidth]{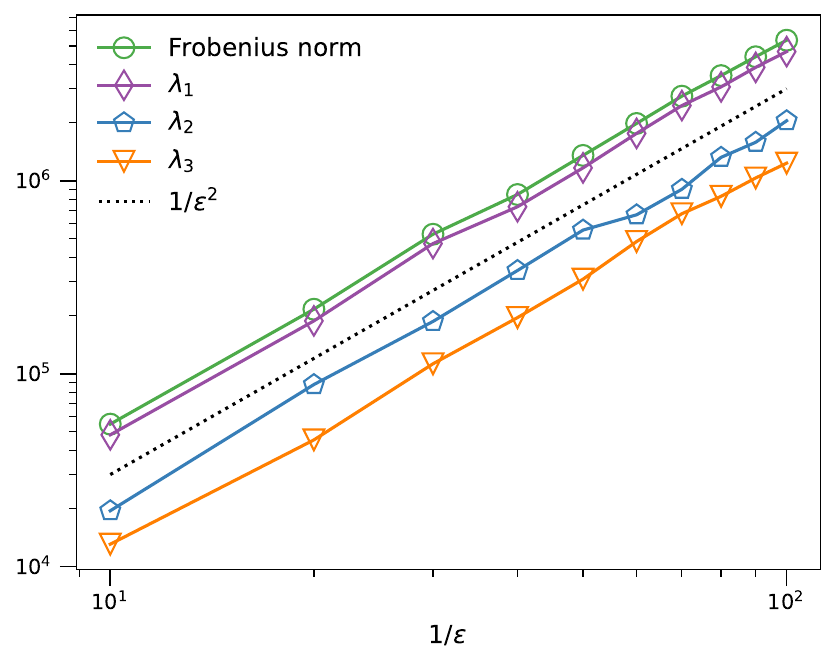}
(b)
 \includegraphics[width=0.45\textwidth]{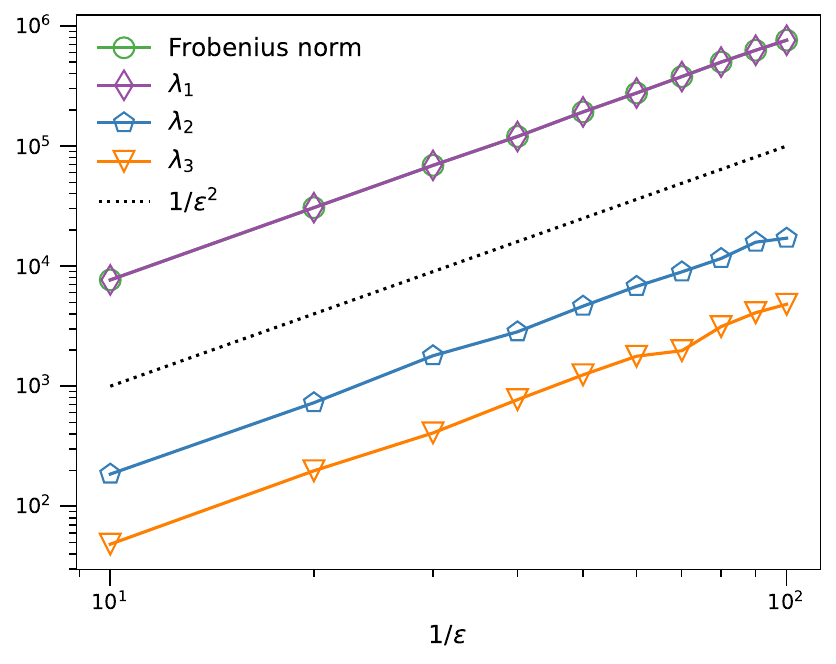}
\caption{Scaling of the Frobenius norm of the NTK matrix and the first three principal eigenvalues
(a) at initialization and (b) after training with Adam optimizer.
} 
\label{fig:ntk_scaling}
\end{figure}

Recall that the Frobenius norm of a square matrix is equal to the Euclidean norm of its singular values. 
Since $\Keps_{uu}$ is a symmetric matrix, its singular values equal the absolute value of its
(real valued) eigenvalues. Hence, Theorems \ref{thm:main_thm} and \ref{thm:second_thm} imply that 
$\sum_{i=1}^{N_c} (\lambda^{\epsilon}_i)^2 \to \infty$ as $\epsilon$ vanishes, 
where $\lambda^{\epsilon}_i$ are the eigenvalues of $\Keps_{uu}$.  Figures \ref{fig:ntk_scaling}(a) 
and \ref{fig:ntk_scaling}(b) further suggest that, for the BVP just described, the first few 
principal eigenvalues (denoted $\lambda_1, \lambda_2$, and $\lambda_3$) 
 also increase at a 
rate proportional to $1/\epsilon^2$, both at initialization and after training with the Adam optimizer.  
In contrast, the spectral radius of the $K_{bb}$ subblock (not shown) is essentially independent
of $\epsilon$. 

\begin{figure}[h]       
    \centering
(a)
 \includegraphics[width=0.45\textwidth]{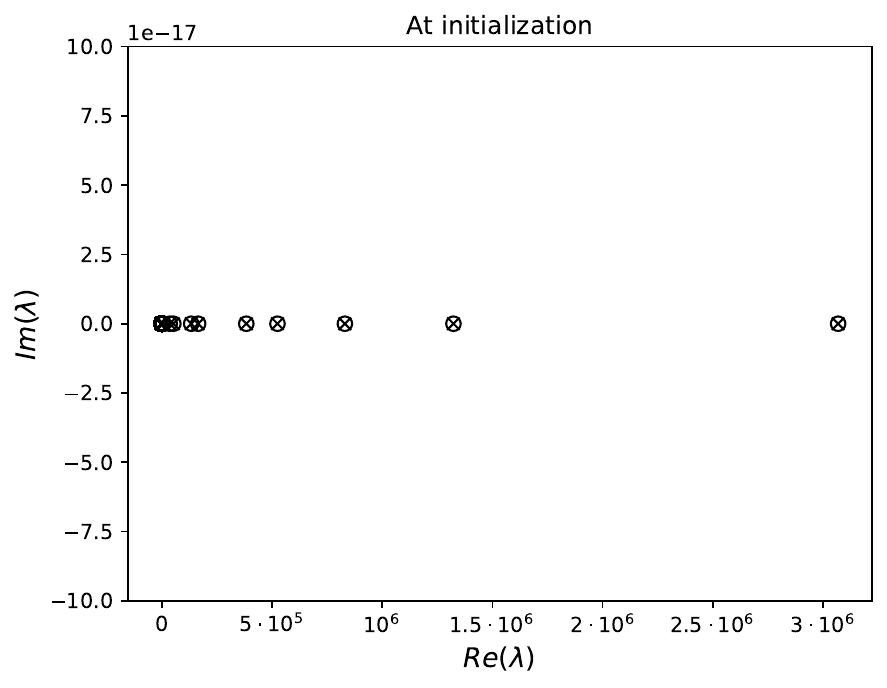}
(b)
 \includegraphics[width=0.45\textwidth]{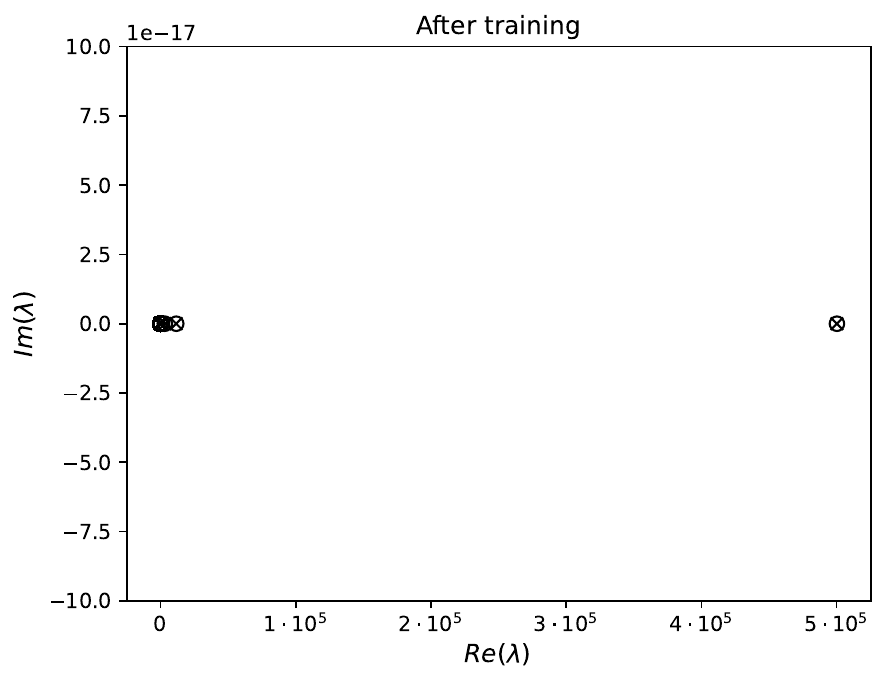}
\caption{ 
The eigenspectrum of $\Keps_{uu}$ (a) at initialization and (b) after training. 
} 
\label{fig:ntk_spectra}
\end{figure}

For the particular case $\epsilon =1/80$, Figure \ref{fig:ntk_spectra} shows the full spectra
of $\Keps_{uu}$ both at initialization and after training. 
At initialization, there are a handful of large eigenvalues, but the rest are clustered around
the origin; after training, there is a single large eigenvalue well-separated from the cluster
near the origin. In both cases 
the ratio $\lambda^{\epsilon}_{\rm max}/\lambda^{\epsilon}_{\rm min}$ is quite large, 
indicating that \eqref{eq:y_evolution} is a stiff initial value problem.

\subsection{Learning a two-scale function}\label{subsec:twoscale_learning}
Consider next the two-scale function 
\begin{equation}\label{eq:two_scale_fn}
\ueps(x) = \sin(2x) + \epsilon \sin(x/\epsilon) - \frac{\epsilon}{\pi} \sin(\pi/\epsilon) \, x.
\end{equation}
This function satisfies the Poisson BVP 
\begin{equation}\label{eq:poisson}
-\frac{d^2}{dx^2} \ueps(x) = 4\sin(2x) + \frac{1}{\epsilon} \sin(x/\epsilon), \qquad -\pi < x < \pi
\end{equation}
with homogeneous Dirichlet boundary conditions 
$\ueps(-\pi) = \ueps(\pi) = 0$;  
it additionally satisfies the elliptic, Darcy BVP given by 
\begin{equation}\label{eq:darcy}
-\frac{d}{dx}\Big(a(x/\epsilon) \frac{d}{dx} \ueps(x) \Big) = f^{\epsilon}(x),
\end{equation}
for $x \in (-\pi,\pi)$, again with homogeneous Dirichlet conditions, where  
\begin{equation*}
a(x/\epsilon) = \big( 2.1 + 2\sin(x/\epsilon)\big)^{-1}
\end{equation*}
and the expression for the right-hand side forcing $f^{\epsilon}(x)$ is listed in Appendix \ref{sec:appendix4}.

We describe three different neural network approximations to the two-scale function $\ueps(x)$ 
by:  
\begin{itemize}
\item[(1)] regression (i.e. supervised learning),
\item[(2)] a PINN solver for the Poisson problem \eqref{eq:poisson}, and
\item[(3)]  a PINN solver for the Darcy problem \eqref{eq:darcy}.
\end{itemize}
We first report the results for each case when $\epsilon=1/32$ and then 
repeat the experiments in Section \ref{subsec:decrease_epsilon} at smaller
values of $\epsilon$ to observe the behavior of each approximation as a function 
of increasing scale separation.

 For $\epsilon = 1/32$, the first two methods listed above generate satisfactory approximations
to \eqref{eq:two_scale_fn}, indicating that (a good approximation to) the target function lives
in the ``span'' of the neural network used; the third method does not. 
All of the numerical 
tests are implemented in the open-source DeepXDE package \citep{lu2021deepxde}, 
and in each case, the neural network weights and biases are randomly
 initialized with the Glorot distribution \citep{glorot2010understanding}. 

Consider first approximating $\ueps(x)$ with a neural network  
by regression, which we denote as $u_R(x;\theta)$. For $\{x_i\}_{i=1}^N \in [-\pi, \pi]$, the 
network parameters are determined by minimizing
\begin{equation}\label{eq:regression_loss}
\mathcal{L}^{\epsilon}(\theta) = \frac{1}{N} \sum_{i=1}^N \big( u_{\rm R}(x_i;\theta) - \ueps(x_i)\big)^2.
\end{equation}
We set $\epsilon =1/32$ and use a neural network with hyperbolic tangent activation 
functions and four hidden layers, each of width $d=40$. Using $N = \ceil{2\pi/(\epsilon/2)} = 403$ 
equispaced training points, the network is trained with the Adam \citep{kingma2014adam}
and L-BFGS \citep{byrd1995limited} optimizers for ten independent trials. In each test, Adam is 
used for $20,000$ and $10,000$ iterations with learning rates
of $\eta =10^{-2}$ and $\eta=10^{-3}$, respectively, and L-BFGS is then used, on average, for 
approximately $3600$ steps. 

Let $ \expval \, u_{\rm R}(x;\theta)$ denote the neural network averaged over all ten trials. 
Taking $\{x_i\}_{i=1}^{1000}$ to be a set of equispaced points from $-\pi$ to $\pi$, 
Figures \ref{fig:regression_and_poisson}(a) and (b) show the resulting approximation and error, respectively. 
The maximum error is approximately
$$ 
 \max_{x_i} \left| e_{\rm R} \right| := \max_{x_i} \big| \ueps(x_i) - \expval\, u_{\rm R}(x_i; \theta)\big| \approx 2.65\cdot 10^{-3}, 
$$
and the average variance in the error over the interval $[-\pi,\pi]$ is about $3\cdot 10^{-7}$.

\begin{figure}[h]       
    \centering
(a)
 \includegraphics[width=0.45\textwidth]{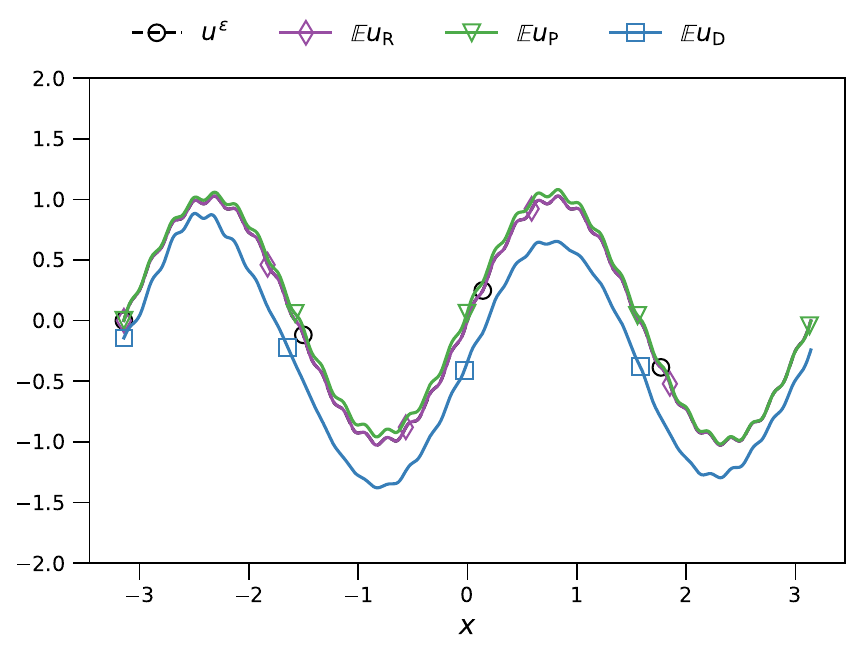}
(b)
 \includegraphics[width=0.45\textwidth]{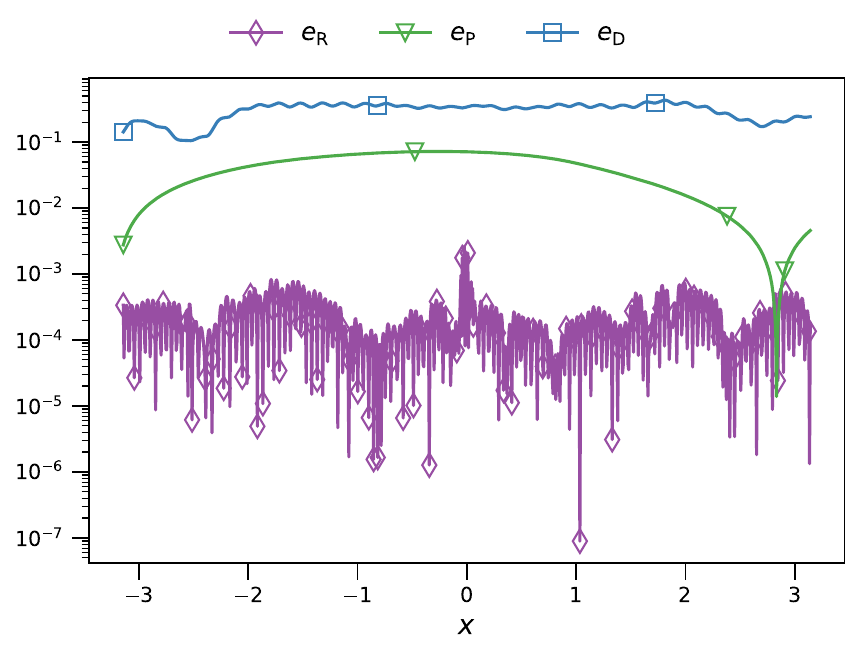}
\caption{(a) For $\epsilon =1/32$, the two-scale function $\ueps(x)$ given by \eqref{eq:two_scale_fn} and its approximations
by regression ($\mathbb{E}\, u_{\rm R}$), a Poisson PINN ($\mathbb{E}\,u_{\rm P}$) and a Darcy PINN ($\mathbb{E}\,u_{\rm D}$); in each case the neural network
 contains four hidden layers of width $d=40$. 
(b) The absolute value of corresponding generalization 
errors $e_{\rm R}$, $e_{\rm P}$ and $e_{\rm D}$ at 1000 equispaced points from $-\pi$ to $\pi$. 
} 
\label{fig:regression_and_poisson}
\end{figure}

Next, consider approximating the solution to the Poisson problem \eqref{eq:poisson} with a PINN $u_{\rm P}(x;\theta)$. 
The loss function is 
\begin{equation}\label{eq:loss_poissonn}
\losseps(\theta) = \frac{1}{N_c} \sum_{i=1}^{N_c} \frac12 \Big( \frac{d^2}{dx^2} u_{\rm P}(x_i;\theta) + (4\sin(2x_i) + \frac{1}{\epsilon} \sin(x_i/\epsilon))\Big)^2
+ \frac12 \Big( [u_{\rm P}(-\pi; \theta)]^2 + [u_{\rm P}(\pi; \theta)]^2 \Big).
\end{equation}
For $\epsilon=1/32$ and same neural network as before, we train with $N_c = 1024$ equispaced collocation 
points for ten independent trials. The training schedule consists of 
$10,000$, $10,000$, $20,000$, and $20,000$ iterations of Adam with learning
rates of $\eta = 10^{-4}$, $\eta = 10^{-5}$, $\eta = 10^{-6}$, and $\eta = 10^{-7}$, respectively, and then 
L-BFGS for a few thousand iterations, on average. 

 Figures \ref{fig:regression_and_poisson}(a) and (b) show the resulting approximations 
and the generalization errors of the neural networks after training. 
Taking again $\expval\, u_{\rm P}(x;\theta)$ to be the network averaged over all the trials, 
the maximum generalization error is 
\begin{equation*}
 \max_{x_i} \left| e_{\rm P} \right| := \max_{x_i} \left| \ueps(x_i) - \expval\, u_{\rm P}(x_i;\theta)\right| \approx 7.20\cdot 10^{-2}, 
\end{equation*}
while the average variance across the interval is about $1.12\cdot 10^{-1}$. 

Finally, consider the PINN approximation $u_{\rm D}(x;\theta)$ of the Darcy problem \eqref{eq:darcy}, where the loss function 
is 
\begin{equation}\label{eq:loss_darcy}
\losseps(\theta) = \frac{1}{N_c} \sum_{i=1}^{N_c} \frac12 \Big[ \frac{d}{dx}\Big(a(x_i/\epsilon) \frac{d}{dx}u_{\rm D}(x_i;\theta)\Big) + f^{\epsilon}(x_i) \Big]^2
+ \frac12 \Big( [u_{\rm D}(-\pi; \theta)]^2 + [u_{\rm D}(\pi; \theta)]^2 \Big).
\end{equation}
Once again, $\epsilon=1/32$, and the network architecture and number of collocation points is the same as 
the Poisson case. The training schedule consists of $10,000$, $10,000$, $20,000$, and $20,000$ iterations of Adam with learning
rates of $\eta = 10^{-3}$, $\eta = 10^{-4}$, $\eta = 10^{-5}$, and $\eta = 10^{-6}$, respectively, and then 
L-BFGS for a few thousand iterations, on average. The maximum generalization error 
\begin{equation*}
  \max_{x_i} \left| e_{\rm D} \right| := \max_{x_i} \left| \ueps(x_i) - \expval\, u_{\rm D}(x_i;\theta)\right| \approx 4.31\cdot 10^{-1}, 
\end{equation*}
which is about 40\%, and the average variance across the interval is nearly as large--about $2.83\cdot 10^{-1}$.
From Figures \ref{fig:regression_and_poisson}(a) and (b) it can be seen that the error contains both low and
high frequency components.

\subsection{Results for decreasing $\epsilon$}\label{subsec:decrease_epsilon}
We now consider the results for the three different neural network approximations to \eqref{eq:two_scale_fn}
 at the values $\epsilon = 1/64$ and $\epsilon = 1/128$. 

Consider first the case of simple $L^2$ regression (i.e.\ minimizing \eqref{eq:regression_loss}).  
For both values of $\epsilon$, 
$N = \ceil{2\pi/(\epsilon/2)}$ equispaced training points are used for a network with four hidden layers, 
each of width $d=40$. Ten independent trials are conducted where 
the network is trained with the Adam optimizer
for 20,000 and 10,000 iterations at a learning rate of $\eta = 10^{-3}$ and $\eta= 10^{-4}$, respectively. The 
generalization error is tested at 1000 equispaced points in $[-\pi,\pi]$. 

The results (not shown) are consistent with the frequency principle, i.e.\ the known spectral bias of neural networks in the 
regression setting \cite{xu2020}. The $\infty$-norm of the average generalization error is about $1.71\cdot 10^{-2}$
and $8.88\cdot 10^{-3}$ for $\epsilon=1/64$ and $1/128$, respectively, while in both cases the maximum pointwise 
variance is relatively small, on the order of $10^{-6}$. The Fourier transform of the error in both cases is essentially zero everywhere 
except at the $1/\epsilon$ frequency, as expected. The techniques proposed in \cite{yang2021overcoming} for example 
could be used to mitigate this issue in this setting, if desired. 

For the Poisson and Darcy PINNs, we consider a neural network with eight hidden layers, each 
of width $d=40$. The training schedule consists of 2000 iterations of Adam with a learning rate of $\eta = 10^{-6}$,
5000 iterations with $\eta = 10^{-7}$, 10000 iterations with $\eta=10^{-8}$ and then L-BFGS. 
 The pointwise generalization errors for the two cases $\epsilon=1/64$ and $\epsilon=1/128$ are 
shown in Figures \ref{fig:poisson_smaller_eps} and \ref{fig:darcy_smaller_eps} in Appendix \ref{sec:appendix5} 
for the Poisson and Darcy PINNs, respectively. 
The $\infty$-norms of the average generalization error for each $\epsilon$ case are reported in Tables \ref{table:errors_poisson}; 
and \ref{table:errors_darcy} for the Poisson and Darcy PINNs, respectively. Also reported are the 
maximum average errors on the boundary $\partial \Omega = \{-\pi,\pi\}$, as well as the $\infty$-norm of the
 variance in the generalization error throughout $\Omega$ and on $\partial \Omega$. 

Neither PINN does particularly well, as expected from the theory developed here for the Darcy case and the
analysis in \cite{wang2021eigenvector,wang2022_JCP} for the Poisson case. 
Although there are some similarities in the error profiles--for example, the variance over the independent trials
becomes larger for both PINN cases as $\epsilon$ decreases--there are some important differences to highlight 
as well. 

For the Poisson PINN, the average generalization error is approximately the same at both values of $\epsilon$, and 
in particular, the PINN satisfies the homogeneous Dirichlet BC to about three digits of accuracy for each $\epsilon$. To ensure
this, we had to set the boundary weight $\lambda_{\rm BC} = 10$ (this value was kept consistent for all the cases 
listed in Tables \ref{table:errors_poisson} and \ref{table:errors_darcy}). In contrast, the average 
error in the Darcy case increases for decreasing $\epsilon$, and in particular, the error at the boundaries
increases. These results are consistent with the discussion in Section \ref{subsec:discussion}, i.e.\ that the 
oscillatory coefficient within the differential operator for the Darcy PINN worsens the imbalance 
between the two loss terms $\mathcal{L}_{\rm PDE}$ and $\mathcal{L}_{\rm BC}$ already known to exist for Poisson PINNs.
\begin{table}
  \begin{center}
  \def~{\hphantom{0}}
    \begin{tabular}{c | c | c | c | c  }
\hline
     $\epsilon$ &  $\|e_{\rm P}\|_{L^{\infty}(\Omega)}$ & $ \| \textrm{Var}_{\rm P}\|_{L^{\infty}(\Omega)}$  &  $\|e_{\rm P}\|_{L^{\infty}(\partial \Omega)}$
                &  $\| \textrm{Var}_{\rm P}\|_{L^{\infty}(\partial \Omega)}$     \\
\hline
$1/64$    & $1.53\cdot 10^{-1}$ & $5.42\cdot 10^{-2}$ & $4.21\cdot 10^{-3}$ & $1.67\cdot 10^{-5}$  \\
$1/128$   & $1.60\cdot 10^{-1}$ & $9.83\cdot 10^{-2}$ & $4.00\cdot 10^{-3}$ & $5.26\cdot 10^{-5}$  \\
    \end{tabular}
    \caption{$\infty$-norm of the generalization error for the Poisson PINN throughout the domain $\Omega = [-\pi,\pi]$ and 
on the domain boundary $\partial \Omega$ for $\epsilon \in \{1/64,1/128\}$. Also shown are the maximum variances
 in the generalization errors, both in $\Omega$ and on the boundary $\partial \Omega$. 
}
  \label{table:errors_poisson}
  \end{center}
\end{table}

\begin{table}
  \begin{center}
  \def~{\hphantom{0}}
    \begin{tabular}{c | c | c | c | c  }
\hline
     $\epsilon$ &  $\|e_{\rm D}\|_{L^{\infty}(\Omega)}$ & $ \| \textrm{Var}_{\rm D}\|_{L^{\infty}(\Omega)}$  &  $\|e_{\rm D}\|_{L^{\infty}(\partial \Omega)}$
                &  $\| \textrm{Var}_{\rm D}\|_{L^{\infty}(\partial \Omega)}$     \\
\hline
$1/64$    & $1.01\cdot 10^{-1}$ & $8.43\cdot 10^{-2}$ & $7.84\cdot 10^{-2}$ & $8.34\cdot 10^{-2}$  \\
$1/128$   & $3.80\cdot 10^{-1}$ & $3.06\cdot 10^{-1}$ & $3.11\cdot 10^{-1}$ & $2.99\cdot 10^{-1}$  \\
    \end{tabular}
    \caption{
$\infty$-norm of the generalization error for the Darcy PINN throughout the domain $\Omega = [-\pi,\pi]$ and 
on the domain boundary $\partial \Omega$ for $\epsilon \in \{1/64,1/128\}$. Also shown are the maximum variances
 in the generalization errors, both in $\Omega$ and on the boundary $\partial \Omega$.
}
  \label{table:errors_darcy}
  \end{center}
\end{table}

%
%
%
%
%
%
%

\subsection{Rescaling the loss function}\label{subsec:rescale_test}
As mentioned in Section \ref{subsec:discussion}, it is reasonable to try mollify the imbalance between the two parts of the Darcy PINN loss function 
by rescaling the PDE residual portion $\mathcal{L}_{\rm PDE}$ by $\epsilon^2$, as this cancels the $1/\epsilon^2$ 
factor in the NTK matrix $\Keps_{uu}$ subblock. 
Here we test this idea at a sequence of $\epsilon$ values.  

We consider neural networks with three hidden layers each of width $d=70$. As before, we use hyperbolic 
tangent activation functions with $N_c = 1024$ collocation points. The training schedule consists of
ten thousands iterations each of the learning rates $\eta = 10^{-3}$, $10^{-4}$, $10^{-5}$ and $10^{-6}$, 
and then, on average a few thousand iterations of L-BFGS. As before, we conduct ten independent trials 
and then average the results. 

In particular we consider learning the solution function
$$
\ueps(x) = e^{-(x-1)^2} + \epsilon \sin(x/\epsilon) + c_1 x + c_2
$$
where $c_1$ and $c_2$ are chosen such that $\ueps$ satisfies homogeneous Dirichlet boundary conditions 
on the domain $[-\pi,\pi]$; with
$$
\aeps(x) = 1.1 + \cos(x/\epsilon)\cos(2x/\epsilon),
$$
we set the forcing function $f^{\epsilon}$ in the Darcy problem such that $\ueps$ indeed solves
\begin{equation}\label{eq:darcy_problem_case2}
-\frac{d}{dx} \Big( \aeps(x) \frac{d}{dx} \ueps(x) \Big) = f^{\epsilon}(x) .
\end{equation}

The pointwise generalization errors for three values of $\epsilon$, as well as the magnitude of their Fourier coefficents,
are shown in Figure \ref{fig:rescaled_darcy_cases} in Appendix \ref{sec:appendix5}. 
Compared to the Darcy PINNs without rescaling, the homogeneous
Dirichlet conditions are much better satisfied. For the $\epsilon = 1/36$, the average error on the boundary is on the order of $10^{-3}$, 
consistent with the Poisson cases in the previous section. For $\epsilon=1/72$ and $\epsilon = 1/144$, the accuracy 
is lower (on the order of $10^{-1}$), but the variance is quite low, in contrast to the huge variance observed in the 
unscaled Darcy PINNs. 
In each case, the error in Fourier space has both a low frequency peak concentrated in the first few wavenumbers and 
a peak at the $1/\epsilon$ frequency; on average the errors in these peaks increase for decreasing $\epsilon$. 
In particular, at the smallest value $\epsilon=1/144$, an additional low frequency peak emerges around $k=16$. 

As discussed extensively in \cite{wang2022_JCP} and mentioned in Section \ref{subsec:discussion}, both the magnitude 
and the distribution of the eigenvalues of the 
NTK matrix subblocks are closely connected with PINN performance. As shown in Figure \ref{fig:ntk_spectra}, there is a 
large discrepancy between the principal eigenvalue of the $\Keps_{uu}$ subblock and the rest of the spectrum, which
is a characteristic sign that the gradient flow dynamics \eqref{eq:grad_flow} are stiff. Although the rescaling introduced
by setting $\lambda_{\rm PDE} = \epsilon^2$ changes the magnitude of the eigenvalues of the $\Keps_{uu}$ subblock, it of course
does not change their distribution, and we attribute the increasing errors at decreasing $\epsilon$ values to the discrepancy 
in the spectrum.

\section{Conclusions}
\label{sec:conclusions}
Physics informed neural networks have demonstrated success
in a wide variety of problems in scientific computing 
\citep{karniadakis2021physics}, however, they can sometimes struggle to 
approximate solutions to differential equations with multiscale 
features. 
Athough PINN convergence theory guarantees that minimizers of (regularized) loss
functions can in principle well approximate solutions to elliptic problems,
in practice, achieving those minimizers is difficult
for boundary value problems with highly oscillatory 
coefficients present in the differential operator. 

We show that for a class of linear, multiscale elliptic equations in divergence form, 
the Frobenius norm of the neural tangent kernel matrix associated to the PINN 
becomes unbounded as the characteristic wavelength of the oscillations vanishes. 
Numerical examples illustrate that during training, the ordinary 
differential equation that governs the evolution of the PDE residuals
becomes increasingly stiff as $\epsilon \downarrow 0$, translating to 
poor PINN training behavior.  

%

The present work considers standard, fully connected 
neural network architectures; 
it may be possible in future research to  
develop alternative PINN architectures that are specifically 
adapted to the problem. 
For example, different activation functions with problem-tuned inductive biases 
can be used, as suggested in \cite{ziyin2020neural} for periodic functions. 
\cite{li2020multiscale} also proposed smooth and localized activation functions
for multiscale elliptic problems in the context of the Deep Ritz method \cite{yu2018deep}. 
In this context, it may be beneficial to develop a theory for the Neural Tangent Kernel 
matrix for energy-minimization based neural network methods for solving multiscale PDEs, as
the derivative of the oscillatory coefficient $\aeps$ does not appear, in contrast to PINNs.

Future work may also benefit from 
more sophisticated learning techniques
beyond (stochastic) gradient descent, 
for example genetic algorithms \citep{mirjalili2019genetic}.
Another possibility is to employ continuation methods \citep{allgower2003introduction},
i.e.\ ``curriculum regularization'' \citep{krishnapriyan2021characterizing}.

Recent results establishing a ``frequency principle'' 
in machine learning motivated this study;  
during training, neural networks tend to learn their target 
functions from low to high frequency \citep{luo2021}. See also the work by \citet{wang2021eigenvector} 
for spectral analysis of physics informed neural networks applied to Poisson problems 
 with multiscale forcing.
An interesting avenue for future study is to characterize the spectral behavior during training
 of deep learning methods for learning operators; see for example
\citep{li2020fourier,khoo2021solving,bhattacharya2021,lu2021learning,zhang:2022}.
In the context of the present work, this would be the nonlinear map from the 
problem data $\aeps$ and $f$ in \eqref{eq:darcys_law} to the solution $\ueps$. 
Finally, we note many other novel neural network approaches to multiscale elliptic problems
not discussed here have recently been proposed, for example 
\citep{han2022neural,leung2021nh,fabian2022operator}.
We hope some of the results presented here might 
inspire insight into these exciting approaches.

\subsection*{Acknowledgments}
The authors thank Andrea Bertozzi for organizing the UCLA Computational and Applied
Mathematics REU site where much of this work took 
place. The authors also thank Chris Anderson and Michael Murray for 
helpful discussions. 

\bibliographystyle{tmlr}
\bibliography{references}


\appendix
\section{Description of PINN approximation for multiscale Poisson problem}
\label{sec:appendix0}
Here we give a more complete description of the example from Section \ref{sec:motivation} 
of a PINN approximation to the Poisson boundary value problem \eqref{eq:1d_poisson} with 
homogeneous Dirichlet boundary conditions. 

The PINN solution is computed with a fully connected neural network with four hidden 
layers of sixty nodes each and hyperbolic tangent activation functions. The training consisted
of $100$, $1,000$, $15,000$, $30,000$ and $10,000$ iterations with the Adam optimizer
with learning rates of $\eta=10^{-2}$, $\eta=10^{-3}$, $\eta=10^{-4}$, $\eta=10^{-5}$, and 
$\eta=10^{-6}$, respectively. We note, however, using a constant value of $\eta=10^{-3}$ throughout
training produced results qualitatively similar to those in Figure \ref{fig:freq_principle_1d_poisson}; 
i.e. the observed frequency principle was fairly robust to changes in the learning rate. 
As mentioned in the description in Section \ref{sec:motivation}, for legibility, the results shown in 
 Figure \ref{fig:freq_principle_1d_poisson} are that of the rolling average (of 50 values) to the trajectories
of the error in the Fourier coefficients. 

Finally, we mention that at the end of the training process, the $\infty$-norm error between 
the true solution and the neural network approximation is about $1.26\cdot 10^{-2}$.

\section{Proof of Lemma \ref{lem:ntk}}
\label{sec:appendix1}

\begin{proof}                    
Computing first the entries of the gradient of the loss function \eqref{eq:loss_eps} 
gives 
\begin{equation}\label{eq:lem1}
\pderiv{\losseps}{\theta_l} = \frac{1}{N_c} \sum_{j=1}^{N_c} r_{\rm pde}(x_j; \theta) \pderiv{}{\theta_l} \Leps u(x_j;\theta)  
+ \frac{\lambda}{N_b} \sum_{j=1}^{N_b} r_{\rm b}(s_j; \theta) \pderiv{}{\theta_l} u(s_j;\theta).
\end{equation}
Next, compute
\begin{align}
\frac{d}{dt} r_{\rm pde}(x_i;\theta(t)) &=\frac{d}{dt}\left(\Leps u(x_i;\theta(t)) - f(x_i) \right) 
= \sum_{l=1}^{N_p} \pderiv{}{\theta_l} \Leps u(x_i;\theta) \frac{d \theta_l}{dt} \nonumber \\
&= -\sum_{l=1}^{N_p} \pderiv{}{\theta_l}\Leps u(x_i;\theta) \pderiv{\losseps}{\theta_l} , \label{eq:lem2}
\end{align}
where the final equality comes from the assumption of gradient flow dynamics \eqref{eq:grad_flow}.
A similar computation gives
\begin{align}
\frac{d}{dt} r_{\rm b}(s_i;\theta(t)) &=\frac{d}{dt} \left( u(s_i;\theta(t))-g(s_i)\right) 
=\sum_{l=1}^{N_p} \pderiv{}{\theta_l} u(s_i;\theta) \frac{d \theta_l}{dt}  \nonumber \\
&= -\sum_{l=1}^{N_p}\pderiv{}{\theta_l} u(s_i;\theta) \pderiv{\losseps}{\theta_l} , \label{eq:lem3}
\end{align}
Inserting \eqref{eq:lem1} into \eqref{eq:lem2} and rearranging the sums gives
\begin{align*}
\frac{d}{dt} r_{\rm pde}(x_i;\theta(t))   
= -\Big(&\sum_{j=1}^{N_c} \Big[ \frac{1}{N_c} \sum_{l=1}^{N_p} 
\pderiv{}{\theta_l}\Leps u(x_i;\theta)\pderiv{}{\theta_l}\Leps u(x_j;\theta)\Big] r_{\rm pde}(x_j;\theta) \nonumber \\
&+\sum_{j=1}^{N_b} \Big[ \frac{\lambda}{N_b}\sum_{l=1}^{N_p} 
\pderiv{}{\theta_l}\Leps u(x_i;\theta)\pderiv{}{\theta_l}u(s_j;\theta)\Big] r_{\rm b}(s_j;\theta)\Big),
\end{align*}
while inserting \eqref{eq:lem1} into \eqref{eq:lem3} similarly results in 
\begin{align*}
\frac{d}{dt} r_{\rm b}(s_i;\theta(t))   
= -\Big(&\sum_{j=1}^{N_c} \Big[ \frac{1}{N_c} \sum_{l=1}^{N_p} 
\pderiv{}{\theta_l}u(s_i;\theta)\pderiv{}{\theta_l}\Leps u(x_j;\theta)\Big] r_{\rm pde}(x_j;\theta) \nonumber \\
&+\sum_{j=1}^{N_b}  \Big[\frac{\lambda}{N_b} \sum_{l=1}^{N_p} 
\pderiv{}{\theta_l} u(s_i;\theta)\pderiv{}{\theta_l}u(s_j;\theta)\Big] r_{\rm b}(s_j;\theta)\Big)
\end{align*}
as desired.

\end{proof}                              

\section{Proof of Theorem \ref{thm:second_thm}}
\label{sec:appendix2}
\begin{proof}                                              
From \eqref{eq:full_nn} we note
that $W^{(1)} \in \mathbb{R}^{d_1\times 1}$, $W^{(l)} \in \mathbb{R}^{d_l \times d_{l-1}}$
for $l = 2,\ldots,\Lambda$, and $W^{(\Lambda+1)} \in \mathbb{R}^{1\times d_{\Lambda}}$, 
while $b^{(l)} \in \mathbb{R}^{d_l}$ for $l=1,\ldots,\Lambda$ and $b^{(\Lambda+1)} \in \mathbb{R}$. 
The fully connected neural network \eqref{eq:full_nn} can then be written as 
\begin{align*}
u(x;\theta) &= \sum_{k_{\Lambda}=1}^{d_{\Lambda}} W^{(\Lambda+1)}_{k_{\Lambda}}
u^{(\Lambda)}_{k_{\Lambda}}(x) + b^{(\Lambda+1)},    \nonumber \\
u^{(\Lambda)}_{k_{\Lambda}}(x) &= 
\sigma\Big(\sum_{k_{\Lambda-1}=1}^{d_{\Lambda-1}} W_{k_{\Lambda}k_{\Lambda-1}}^{(\Lambda)} 
u^{(\Lambda-1)}_{k_{\Lambda-1}}(x) + b_{k_{\Lambda}}^{(\Lambda)} \Big), \qquad 1 \le k_{\Lambda} \le d_{\Lambda}, \nonumber \\
u^{(\Lambda-1)}_{k_{\Lambda-1}}(x) &= 
\sigma\Big(\sum_{k_{\Lambda-2}=1}^{d_{\Lambda-2}} W^{(\Lambda-1)}_{k_{\Lambda-1}k_{\Lambda-2}}
u^{(\Lambda-2)}_{k_{\Lambda-2}}(x) + b^{(\Lambda-1)}_{k_{\Lambda-1}}\Big), \qquad 1 \le k_{\Lambda-1} \le d_{\Lambda-1},   \nonumber \\
&\vdots  \nonumber \\
u^{(2)}_{k_{2}}(x) &= 
\sigma\Big(\sum_{k_1=1}^{d_1} W^{(2)}_{k_{2}k_{1}}
u^{(1)}_{k_1}(x) + b^{(2)}_{k_{2}}\Big), \qquad 1 \le k_2 \le d_2  \nonumber \\
u^{(1)}_{k_{1}}(x) &= 
\sigma\Big(W^{(1)}_{k_{1}}
x + b^{(1)}_{k_{1}}\Big), \qquad 1 \le k_1 \le d_1.     \label{eq:full_nn_index} 
\end{align*}
By linearity, note that
\begin{equation}\label{eq:linearity_again}
\pderiv{}{W_{\gamma}^{(\Lambda+1)}} \Leps u(x;\theta) = \Leps u^{(\Lambda)}_{\gamma}(x)
\end{equation}
for any $1 \le \gamma \le d_{\Lambda}$.

As before, consider a diagonal entry in the $\Keps_{uu}$ subblock 
of the NTK matrix \eqref{eq:Keps}
\begin{equation*} \label{eq:Kuu_entry_again}
(\Keps_{uu})_{\alpha\alpha}(t) = \frac{1}{N_c} \sum_{\gamma=1}^{N_p} 
\pderiv{}{\theta_{\gamma}} \Leps u(x_{\alpha}; \theta(t))   \pderiv{}{\theta_{\gamma}}\Leps u(x_{\alpha}; \theta(t)).
\end{equation*}
By \eqref{eq:linearity_again}, a sufficient condition for the Frobenius norm of $\Keps$ to blow-up, then, is for
\begin{equation*}\label{eq:easter_egg}
\lim_{\epsilon \downarrow 0} \big(\Leps u^{(\Lambda)}_{\gamma}(x_{\alpha}) \big)^2 = \infty
\end{equation*}
for some collocation point $x_{\alpha}$, some $1 \le \gamma \le d_{\Lambda}$, and $t \in \Upsilon$ 
(here the dependence of $u^{(\Lambda)}_{\gamma}$ on $t$ and $\theta$ is implied). 
Since 
\begin{align}
\big(\Leps u^{(\Lambda)}_{\gamma}(x_{\alpha})\big)^2  = \Big[ 
\Big( a(x_{\alpha}/\epsilon) \frac{d^2}{dx^2}\, u^{(\Lambda)}_{\gamma}(x_{\alpha})\Big)^2 &+ 
\frac{2}{\epsilon} \Big(a(x_{\alpha}/\epsilon) \frac{d^2}{dx^2}\, u^{(\Lambda)}_{\gamma}(x_{\alpha})
\, a'(x_{\alpha}/\epsilon)  \frac{d}{dx}\, u^{(\Lambda)}_{\gamma}(x_{\alpha})\Big)\nonumber  \\
& +  \frac{1}{\epsilon^2} \Big(a'(x_{\alpha}/\epsilon) \frac{d}{dx}\, u^{(\Lambda)}_{\gamma}(x_{\alpha})\Big)^2\Big],
\label{eq:want_to_blowup}
\end{align}
and since $a$ is a periodic $C^1$ function, the desired result will follow if (i) the second derivative of $u_{\gamma}^{(\Lambda)}$ 
is uniformly bounded in $\epsilon$ and (ii) the first derivative is asymptotically larger than $\epsilon$, so that 
\begin{equation}\label{eq:key_play}
\lim_{\epsilon \downarrow 0} \frac{1}{\epsilon} \left| \frac{d}{dx} \, u^{(\Lambda)}_{\gamma}(x) \right| = \infty
\end{equation}
for any $t \in \Upsilon$. 

The first derivative is 
\begin{equation}\label{eq:1st_deriv_u_full}
\frac{d}{dx} u^{(\Lambda)}_{\gamma}(x) = 
\sigma'\Big(\sum_{k_{\Lambda-1}=1}^{d_{\Lambda-1}} W_{\gamma k_{\Lambda-1}}^{(\Lambda)} 
u^{(\Lambda-1)}_{k_{\Lambda-1}}(x) + b_{k_{\Lambda}}^{(\Lambda)} \Big) 
\sum_{k_{\Lambda-1}=1}^{d_{\Lambda-1}}\Big[ W_{\gamma k_{\Lambda-1}}^{(\Lambda)} 
\frac{d}{dx} u^{(\Lambda-1)}_{k_{\Lambda-1}}(x) \Big]
\end{equation}
after applying the chain rule once. 
Using Einstein summation notation for brevity, \eqref{eq:1st_deriv_u_full} simplifies to 
\begin{equation*}\label{eq:1st_deriv_u_full_einstein}
\frac{d}{dx} u^{(\Lambda)}_{\gamma}(x) = 
\sigma'\Big(W_{\gamma k_{\Lambda-1}}^{(\Lambda)} 
u^{(\Lambda-1)}_{k_{\Lambda-1}}(x) + b_{k_{\Lambda}}^{(\Lambda)} \Big) 
\Big[ W_{\gamma k_{\Lambda-1}}^{(\Lambda)} 
\frac{d}{dx} u^{(\Lambda-1)}_{k_{\Lambda-1}}(x) \Big]. 
\end{equation*}
Continuing with the chain rule (and the Einstein notation), we get 
\begin{align}
\frac{d}{dx} u^{(\Lambda)}_{\gamma}(x) 
&= 
\sigma'\Big( W_{\gamma k_{\Lambda-1}}^{(\Lambda)} 
u^{(\Lambda-1)}_{k_{\Lambda-1}}(x) + b_{k_{\Lambda}}^{(\Lambda)} \Big)  \, \times \nonumber \\
&\Big[ W_{\gamma k_{\Lambda-1}}^{(\Lambda)} 
\sigma'\Big(W^{(\Lambda-1)}_{k_{\Lambda-1} k_{\Lambda-2}} 
u^{(\Lambda-2)}_{k_{\Lambda-2}}(x) + b^{(\Lambda-1)}_{k_{\Lambda-1}} \Big) 
\Big[W^{(\Lambda-1)}_{k_{\Lambda-1} k_{\Lambda-2}} 
\frac{d}{dx} u^{(\Lambda-2)}_{k_{\Lambda-2}}(x) \Big]\Big]. \nonumber
\end{align}
From here we can recurse downwards to compute the full derivative as 
\begin{align}
\frac{d}{dx} u^{(\Lambda)}_{\gamma}(x) = 
\sigma'\Big( W_{\gamma k_{\Lambda-1}}^{(\Lambda)} 
u^{(\Lambda-1)}_{k_{\Lambda-1}}(x) + b_{k_{\Lambda}}^{(\Lambda)} \Big)  \, \times& \nonumber \\
\Big[ W_{\gamma k_{\Lambda-1}}^{(\Lambda)} 
\sigma'\Big(W^{(\Lambda-1)}_{k_{\Lambda-1} k_{\Lambda-2}} 
u^{(\Lambda-2)}_{k_{\Lambda-2}}(x) + b^{(\Lambda-1)}_{k_{\Lambda-1}} \Big) 
\Big[W^{(\Lambda-1)}_{k_{\Lambda-1} k_{\Lambda-2}} 
\sigma'\Big(W^{(\Lambda-2)}_{k_{\Lambda-2} k_{\Lambda-3}} 
u^{(\Lambda-3)}_{k_{\Lambda-3}}(x) + b^{(\Lambda-2)}_{k_{\Lambda-2}} \Big) \times& \nonumber \\
\Big[\ldots \Big[W^{(3)}_{k_3 k_2} \sigma'\Big(W^{(2)}_{k_2 k_1}u^{(1)}_{k_1}(x)+b^{(2)}_{k_2} \Big) 
 \Big[W^{(2)}_{k_2 k_1}\sigma'\Big(W^{(1)}_{k_1}x + b^{(1)}_{k_1}\Big)W^{(1)}_{k_1} \Big]\Big]\ldots \Big]\Big]\Big]&.
\label{eq:dudx}
\end{align}
By the uniform boundedness assumption \eqref{eq:assumption2} on the network parameters $\theta$, 
\eqref{eq:dudx} is uniformly bounded in $\epsilon$ from above for any $x \in [a,b]$. As a consequence of the positivity assumption 
\eqref{eq:sigma_assumption}, \eqref{eq:assumption2} further 
implies that each instance of $d\sigma/dx$ in \eqref{eq:dudx} is bounded below by a positive number 
that is independent of $\epsilon$. Note that \eqref{eq:dudx} generally contains sums and products of 
$d\sigma/dx$ and network weights (entries in the $W^{(l)}$ matrices). 
For any $t \in \Upsilon$, the weights are asymptotically larger than $\epsilon^{1/\Lambda}$, so
that 
$$
\lim_{\epsilon \downarrow 0} \Big( \frac{\epsilon}
{
W^{(1)}_{k_1} \, W^{(2)}_{k_2 k_1} \ldots W^{(\Lambda-1)}_{k_{\Lambda-1} k_{\Lambda-2}} W^{(\Lambda)}_{k_{\Lambda} k_{\Lambda-1}} 
} \Big) = 0 
$$ 
for any combination of indices $1\le k_1 \le d_1$, $1 \le k_2 \le d_2$, and so on, meaning 
that \eqref{eq:key_play} indeed holds. 


Ergo, as in the proof of Theorem \ref{thm:main_thm}, as long as $\epsilon$ vanishes monotonically 
to zero in such a way that $a'(x_{\alpha}/\epsilon) \ne 0$ any $\epsilon$, 
\eqref{eq:want_to_blowup} limits to positive infinity as desired.

\end{proof}

\section{Proof of Corollary \ref{cor:spectral_radius}}
\label{sec:appendix3}
\begin{proof}
From elementary linear algebra, we know that the Frobenius norm of the square matrix $\Keps_{uu}(t)$
equals the Euclidean norm of its singular values, and because $\Keps_{uu}(t)$ is a symmetric matrix, 
its singular values equal the absolute value of its (real valued) eigenvalues. If 
 $\{\lambda^{\epsilon}_i(t) \}_{i=1}^{N_c}$ denotes the eigenvalues, then 
$$
\norm{\Keps_{uu}(t)}_{F} =  
\Big(\sum_{i=1}^{N_c} (\lambda^{\epsilon}_i(t))^2 \Big)^{1/2} \le 
\sqrt{N_c} \, \rho(\Keps_{uu}(t))  
$$
giving the result. 
\end{proof}

\section{Boundary value problem forcing functions} 
\label{sec:appendix4}
The right-hand side forcing function $f^{\epsilon}(x)$ for the Darcy BVP 
\eqref{eq:darcy} with homogenenous Dirichlet boundary conditions 
is given by 
$g^{\epsilon}(x)/h^{\epsilon}(x)$, where
\begin{align}
g^{\epsilon}(x) = 10 \, \Big[20 \, \pi + & 168 \, \pi \epsilon \cos\left(x\right) \sin\left(x\right) - 
20  \Big(2 \pi - 4  \pi \cos^2(x) + \epsilon \sin(\pi/\epsilon)\Big) \cos(x/\epsilon)  \nonumber \\
&+ \Big(21  \pi + 160 \, \pi \epsilon \cos\left(x\right) \sin(x)\Big) \sin(x/\epsilon)\Big], \nonumber
\end{align}
and
\begin{equation*}
h^{\epsilon}(x) = 400\,  \pi \epsilon   \sin^2(x/\epsilon) + 840 \, \pi \epsilon   \sin(x/\epsilon) + 441 \, \pi \epsilon .
\end{equation*}

\section{Error plots for decreasing $\epsilon$} 
\label{sec:appendix5}

\begin{figure}[h]       
    \centering
(a)
 \includegraphics[width=0.45\textwidth]{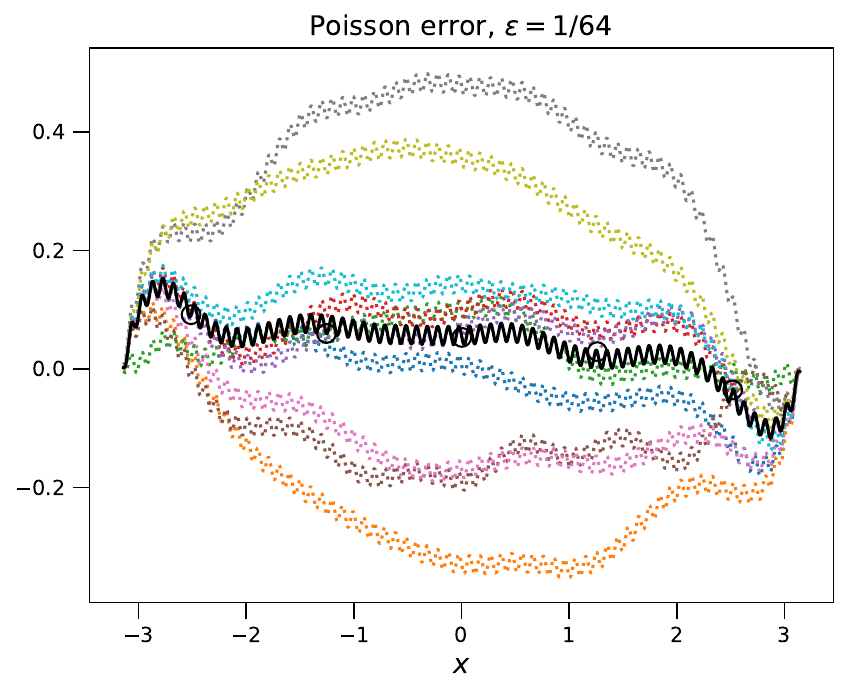}
(b)
 \includegraphics[width=0.45\textwidth]{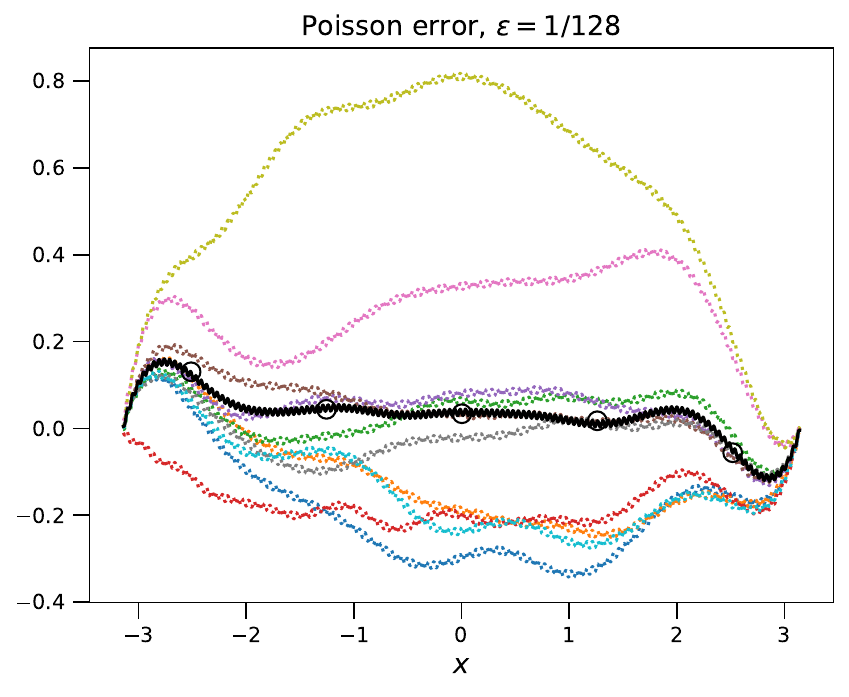}
\caption{
Pointwise generalization errors for the Poisson PINN \eqref{eq:loss_poissonn}. The dotted lines indicate 
independent trial runs, while the solid line with circle markers indicates the 
mean. 
} 
\label{fig:poisson_smaller_eps}
\end{figure}

\begin{figure}[h]       
    \centering
(a)
 \includegraphics[width=0.45\textwidth]{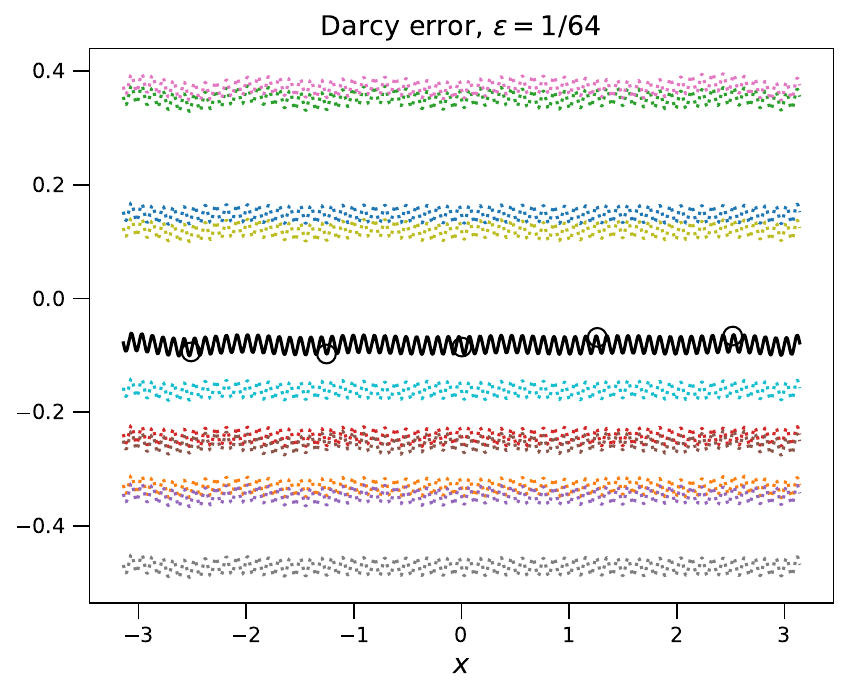}
(b)
 \includegraphics[width=0.45\textwidth]{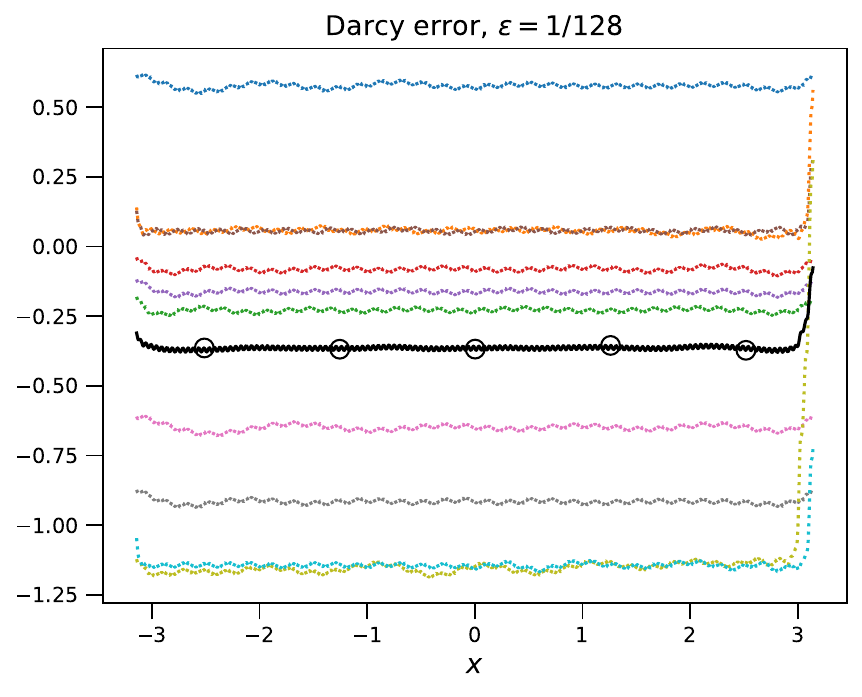}
\caption{
Pointwise generalization errors for the Darcy PINN \eqref{eq:loss_darcy}. The dotted lines indicate 
independent trial runs, while the solid line with circle markers indicates the 
mean.
} 
\label{fig:darcy_smaller_eps}
\end{figure}

\begin{figure}[h]       
    \centering
 \includegraphics[width=0.45\textwidth]{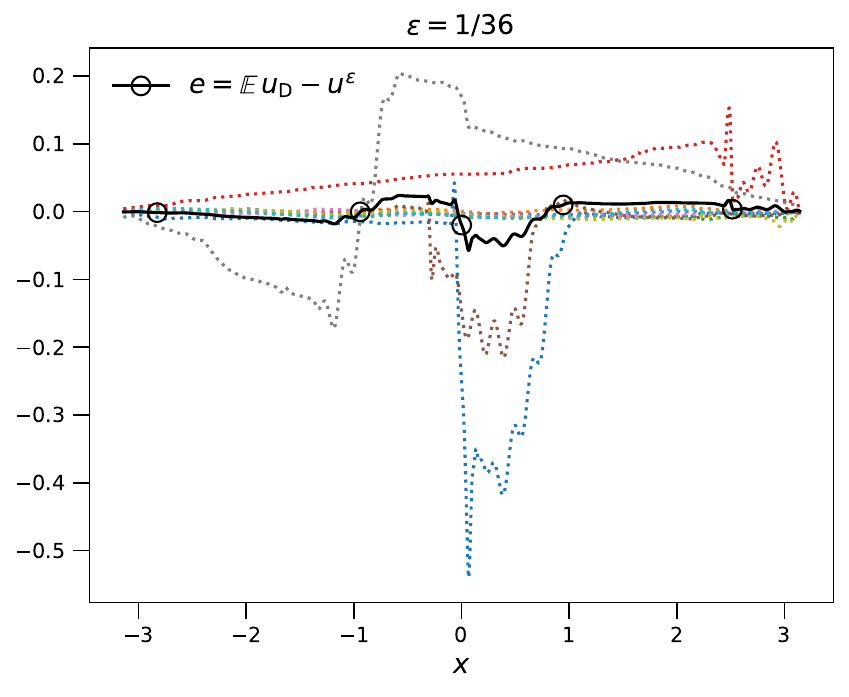}
 \includegraphics[width=0.45\textwidth]{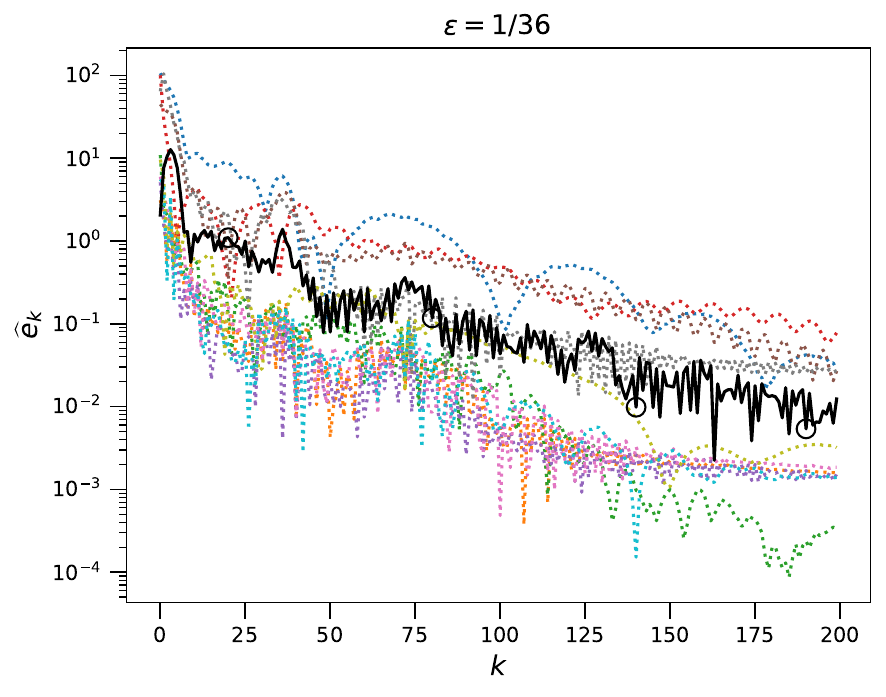}
 \includegraphics[width=0.45\textwidth]{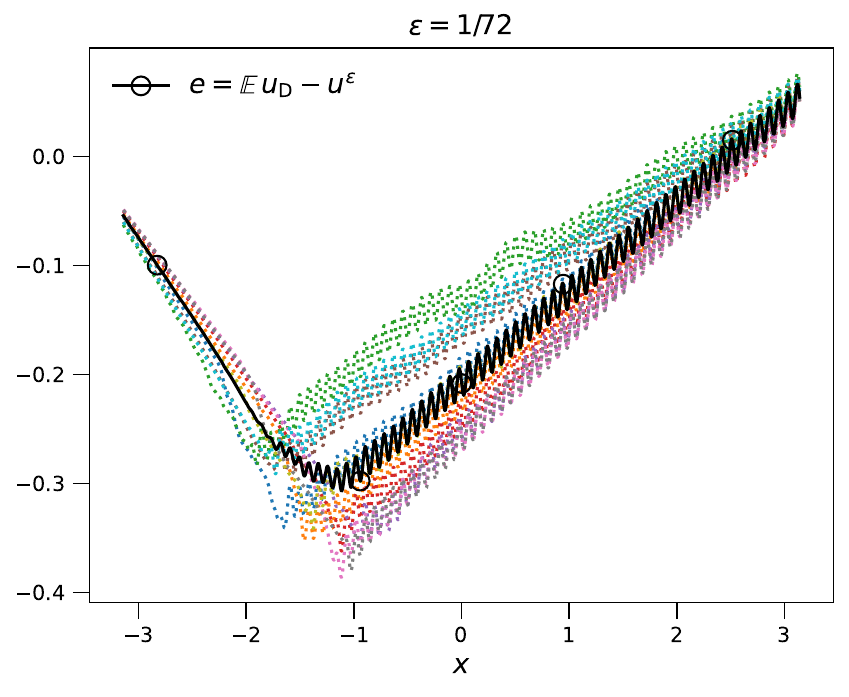}
 \includegraphics[width=0.45\textwidth]{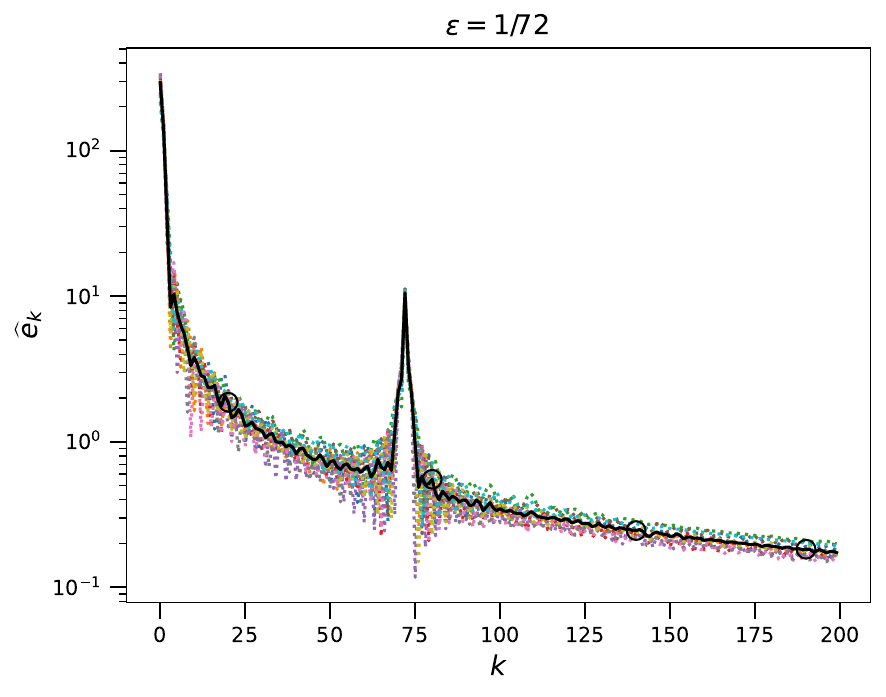}
 \includegraphics[width=0.45\textwidth]{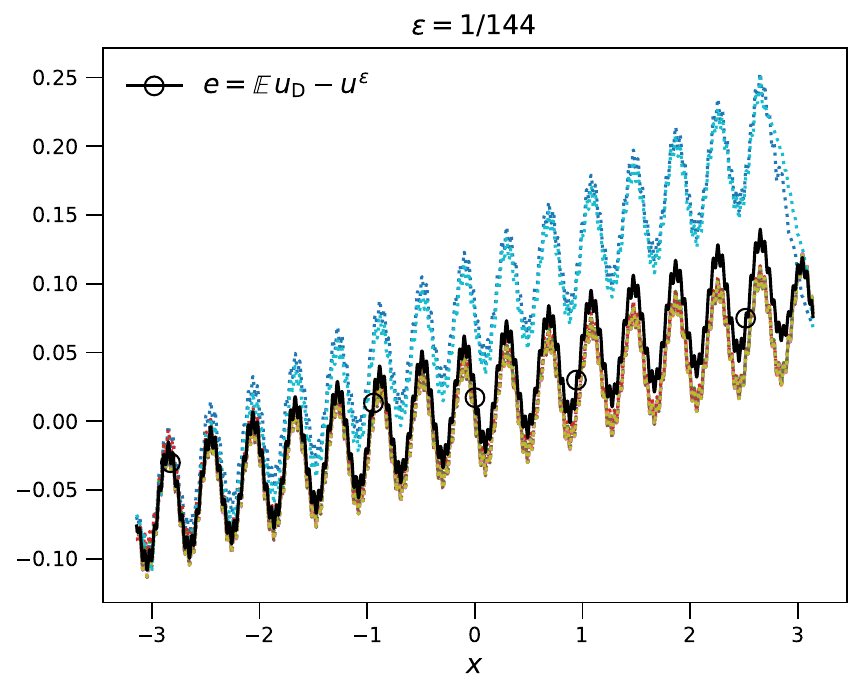}
 \includegraphics[width=0.45\textwidth]{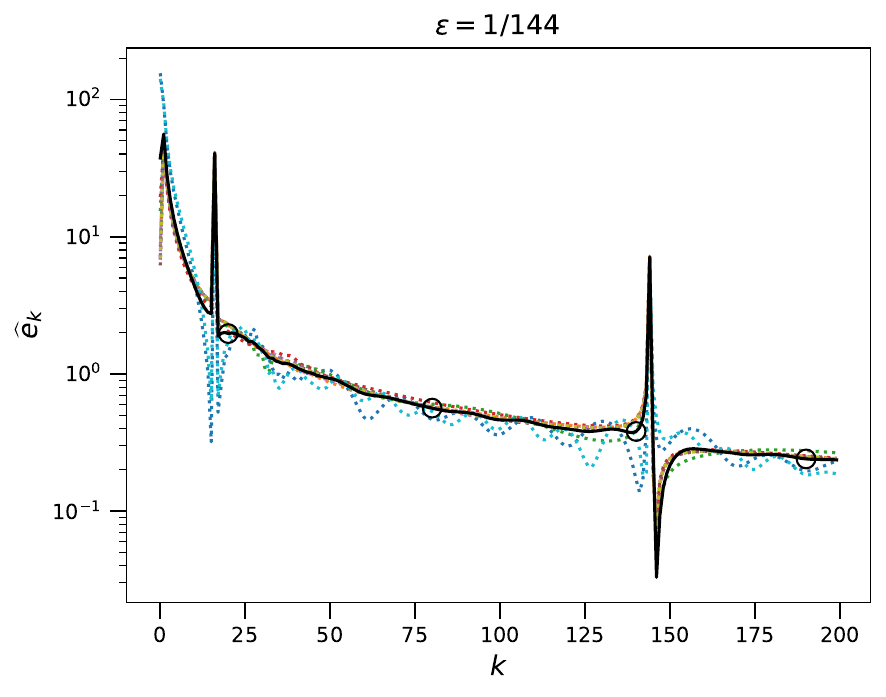}
\caption{
Results at three different $\epsilon$ values for the Darcy PINN associated to \eqref{eq:darcy_problem_case2} 
with the PDE 
residual part of the loss function weighted by $\lambda_{\rm PDE} = \epsilon^2$.
The left columns display the pointwise generalization error on $[-\pi,\pi]$, and the right columns display
the magnitude of the first two hundred Fourier coefficients of the error.
The dotted lines indicate 
independent trial runs, while the solid line with circle markers indicates the 
mean.
} 
\label{fig:rescaled_darcy_cases}
\end{figure}

\end{document}